\documentclass[11pt]{amsart}
\usepackage{amsfonts,amssymb,verbatim,amsmath,amsthm,latexsym,textcomp,amscd}
\usepackage{latexsym,amsfonts,amssymb,epsfig,verbatim}
\usepackage{amsmath,amsthm,amssymb,latexsym,graphics,textcomp,hyperref,enumerate}
\usepackage[capitalize]{cleveref}
\usepackage{eucal}
\usepackage{hyperref}
\usepackage{comment}
\usepackage[utf8]{inputenc}
 \usepackage[T1]{fontenc}
 \usepackage[english, french]{babel}
\usepackage{xcolor}
\usepackage{overpic}
\title[Realisation of bending measured laminations]{Realisation of bending measured laminations by Kleinian surface groups}
\author{Shinpei Baba and Ken'ichi Ohshika}
\address{S.B.: Department of Mathematics, Graduate School of Science, Osaka University, Toyonaka, Osaka, Japan}
\email{baba@math.sci.osaka-u.ac.jp}
\address{K.O.: Department of Mathematics, Faculty of Science, Gakushuin University, Toshima-ku, Tokyo 171-8588, Japan}
\email{ohshika@math.gakushuin.ac.jp}
\newcommand{\complexes}{\mathbb{C}}

\newcommand{\integers}{\mathbb{Z}}

\newcommand{\AH}{\mathsf{AH}}
\newcommand{\QF}{\mathsf{QF}}
\newcommand{\QH}{\mathsf{QH}}

\newcommand{\hyperbolic}{\mathbb{H}}

\newcommand{\Fr}{\operatorname{Fr}}

\newcommand{\len}{\mathrm{length}}
\newcommand{\cc}{\mathcal{CC}}

\newcommand{\ml}{\mathcal{ML}}
\newcommand{\qsp}{\QH_{\lambda_-, \lambda_+}}
 \newcommand{\cml}{\widehat{\mathrm{ML}}}
 \newcommand{\pl}{\mathcal{PML}}
 \newcommand{\plc}{\mathcal{PML}_c}
 \newcommand{\ct}{\widehat{\mathrm{T}}}
\newcommand{\cb}{\widehat{b\circ q}}
\newcommand{\ce}{\mathcal{E}}
\newcommand{\cE}{\widehat{\mathcal E}}
\newcommand{\cF}{\widehat{F}}
\newcommand{\cU}{\widehat{U}}
\newcommand{\ch}{\widehat{H}}

\newcommand{\teich}{\mathcal{T}}
\let \pslc \PSL

\newtheorem{theorem}{Theorem}[section]
\newtheorem{proposition}[theorem]{Proposition}
\newtheorem{lemma}[theorem]{Lemma}
\newtheorem{corollary}[theorem]{Corollary}
\newtheorem{claim}[theorem]{Claim}

\theoremstyle{definition}
\newtheorem{definition}[theorem]{Definition}
\theoremstyle{remark}
\newtheorem{remark}[theorem]{Remark}

\includecomment{comments}
\specialcomment{comments}{\color{red}}{\color{black}}

\begin{document}
\maketitle
%\selectlanguage{french}
%\begin{abstract}
%Bonahon et Otal ont d\UTF{00E9}montr\UTF{00E9} la partie d'existence et partiellement la partie d'unicit\UTF{00E9} de la conjecture des laminations mesur\UTF{00E9}es de plissage au cas des groupes kleiniens librement ind\UTF{00E9}composables g\UTF{00E9}om\UTF{00E9}triquement finis.
%Dans cet article, nous g\UTF{00E9}n\UTF{00E9}ralisons la partie d'existence de leur r\UTF{00E9}sultat au cas des groupes kleiniens isomorphes aux groupes de surface, y compris les g\UTF{00E9}om\UTF{00E9}triquement infinis.
%En ce faisant, nous allons \UTF{00E9}galement d\UTF{00E9}montrer la compacticit\UTF{00E9} de l'ensemble des groupes kleiniens isomorphes aux groupes de surface qui r\UTF{00E9}alisent une donn\UTF{00E9}e quelconque de laminations de plissage et de laminations terminales.
%Notre d\UTF{00E9}monstration est ind\UTF{00E9}pendante du r\UTF{00E9}sultat de Bonahon et Otal.\end{abstract}
\selectlanguage{english}
\begin{abstract}
For geometrically finite Kleinian surface groups, Bonahon and Otal proved the existence part, and partly the uniqueness part of the bending lamination conjecture. 
In this paper, we generalise the existence part to general Kleinian surface groups including geometrically infinite ones. 
Along the way, we also prove the compactness of the set of Kleinian surface groups realising an arbitrarily fixed data of bending laminations and ending laminations. 
Our proof is independent of that of Bonahon and Otal. 
\end{abstract}

\section{Introduction}
The simultaneous uniformisation theorem by Bers \cite{Be} gives a parametrisation of the quasi-Fuchsian space for a closed oriented surface $S$ by the product of two Teichm\"{u}ller spaces $\teich(S) \times \teich(\bar S)$, where $\bar S$ denotes $S$ with its orientation reversed.
This was generalised by the work of Kra, Maskit, Marden and Sullivan, which shows  that for any Kleinian surface group, or more generally for any freely indecomposable Kleinian group $G$, its quasi-conformal deformation space is parametrised by $\teich(\Omega_G/G)$, where $\Omega_G$ is the region of discontinuity of $G$ in the Riemann sphere.
%In both cases, the parametrisation is obtained by considering the conformal structures at infinity, in a sense relying on complex analytic structures rather than hyperbolic structures.

In his lecture notes \cite{ThL}, Thurston considered the convex core of the quotient hyperbolic 3-manifold $\hyperbolic^3/G$ for a Kleinian (surface) group $G$.
He noticed that the boundary of the convex core has two pieces of information: the hyperbolic structure and the bending lamination.
In contrast to the previous work of Bers et al, these are obtained by just considering the quotient hyperbolic manifolds, without looking at structures at infinity.
He seems to have conjectured that both hyperbolic structures  and  bending laminations on the boundaries serve as other kinds of parametrisation of the quasi-Fuchsian space, or more generally, the quasi-conformal deformation space of a freely indecomposable Kleinian group.

As for the first of these two, the hyperbolic structures on the boundaries of convex cores, Sullivan's lemma (see e.g. \cite{EM}) shows that they are within universally bounded distance in the corresponding Teichm\"{u}ller spaces from the conformal structures at infinity on $\Omega_G/G$.
Nevertheless, it is still unknown if they really give a parametrisation of the deformation space.

As for bending laminations, Bonahon-Otal showed in \cite{BO} that every pair of measured laminations on a closed orientable surface $S$, without homotopic components and without compact leaves with weight larger than or equal to $\pi$, can be realised as bending laminations of a quasi-Fuchsian group corresponding to $S$.
In particular, when both of the measured laminations are weighted multi-curves, it was proved that the realising quasi-Fuchsian group is unique (up to conjugation).
They also showed the same result for quasi-conformal deformation spaces of general freely indecomposable geometrically finite groups.
Their result was generalised to freely decomposable Kleinian groups by Lecuire \cite{Lec}.

In this paper, we prove a generalisation of the existence part of this result by Bonahon-Otal to general  Kleinian surface groups including geometrically infinite ones (\cref{main}-(1)), by adding ending laminations  to the data.
We shall furthermore prove the compactness of the set of representations (up to conjugacy) realising  given bending laminations and ending laminations (\cref{main}-(2)).
The proofs of both are different from and  independent of the results by Bonahon-Otal.
On the other hand, we do not have the partial uniqueness result as was given by Bonahon-Otal, for we cannot invoke the theory of cone manifold deformation, whose generalisation to the case of geometrically infinite groups does not exist for the moment.

\sloppy
We shall prove the existence part and the compactness by showing the properness of the following composition of maps in \cref{proper degree 1}.
For a Kleinian surface group $G$, the theory of Bers-Kra-Maskit-Marden-Sullivan gives a parametrisation of the quasi-conformal deformation space $q \colon \teich(\Omega_G/G) \to \QH(G)$, where $\Omega_G$ denotes the region of discontinuity of $G$, $\teich(\Omega_G/G)$ the Teichmüller space of the Riemann surface $\Omega_G/G$, and $\QH(G)$  the space of quasi-conformal deformations of $G$ modulo conjugacy.
Sending each Kleinian group to its bending lamination, we get a map $b \colon \QH(G) \to \ml(\Omega_G/G)$, where $\ml$ denotes the space of measured laminations.
Let $D \subset \ml(\Omega_G/G)$ be the set of measured laminations evidently unrealisable, whose exact definition is given in  \cref{main}.
In this setting, \cref{proper degree 1} states that $q \circ b$ is a proper, degree-1 map to $\ml(\Omega_G/G) \setminus D$.
This in particular says that $q \circ b$ is surjective to $\ml(\Omega_G/G) \setminus D$, and hence we obtain the existence part of the main result.
The compactness part is derived from the properness of the map.

Our proof of \cref{proper degree 1}, is divided into two parts: we shall first show  the properness of the map $b \circ q$  in \cref{properness}, and then that $b\circ q$ has degree 1 in \cref{degree 1}.
In the first part, relying on the analysis of geometric limits, as was given in Ohshika-Soma \cite{OS} and Ohshika \cite{OhT}, we shall show that any sequence going to infinity in $\teich(\Omega_G/G)$  has image under $b \circ q$ which cannot stay in a compact set disjoint from $D$.
In the second part, we shall show that $q \circ b$ can be properly homotoped to a local degree-1 map which is constructed using the earthquake map.
Since the degree is invariant under a proper homotopy, this implies that $b \circ q$ has also degree 1.

The authors would like to express their hearty gratitude to the referee for his/her careful reading and suggestions, due to which we could in particular remove some mistakes in the first version.

\section{Preliminaries}
\subsection{Basics of Kleinian groups}
A {\em Kleinian group} is a discrete subgroup of $\pslc$.
In this paper, we only consider Kleinian groups isomorphic to the fundamental groups of closed orientable surfaces of genus greater than $1$, which we call {\em Kleinian surface groups}.
A Kleinian group acts on the Riemann sphere $\hat \complexes$ by linear fractional transformations and on the hyperbolic space $\hyperbolic^3$ by orientation-preserving isometries.
By considering the Poincar\'{e} model of $\hyperbolic^3$, the Riemann sphere $\hat \complexes$ is regarded as the sphere at infinity of $\hyperbolic^3$.
The action on $\hat \complexes$ is a continuous extension of the action on $\hyperbolic^3$ if we regard $\hat \complexes$ as the points at infinity in this way.
For a Kleinian group $G$, its {\em limit set} $\Lambda_G$ is the closure of the set of fixed points of non-trivial elements of $G$.
The complement of $\Lambda_G$ in $\hat \complexes$ is called the {\em region of discontinuity} of $G$, and is denoted by $\Omega_G$.

The smallest convex subset of $\hyperbolic^3$ containing all geodesics both of whose endpoints at infinity lie on $\Lambda_G$ is called the {\em Nielsen convex hull} and is denoted by $H_G$.
Since $H_G$ is a closed convex subset invariant under $G$, its quotient $H_G/G$ is a closed convex subset of $\hyperbolic^3/G$, which is a 3-submanifold except for the case when $G$ is Fuchsian.
The quotient $H_G/G$ is called the convex core of $\hyperbolic^3/G$ and is denoted by $C(\hyperbolic^3/G)$.
The Kleinian group $G$ is said to be {\em geometrically finite} if $C(\hyperbolic^3/G)$ has finite volume.

\subsection{Geodesic and measured laminations}
\label{mss}
In this section, we consider an orientable surface $S$ which may have punctures but has no boundary.
We fix a complete hyperbolic metric on $S$ which makes punctures cusps.
A {\em geodesic lamination} $\lambda$ on $S$ is a closed subset consisting of disjoint simple geodesics which do not tend to cusps.
A geodesic constituting $\lambda$ is called a {\em leaf}.
A geodesic lamination is said to be {\em minimal} when it does not have a non-empty proper sublamination.
Any geodesic lamination is decomposed into disjoint  finitely many minimal sublaminations, which we call {\em minimal components}, and isolated leaves spiralling around minimal components.
We say that a geodesic lamination is {\em arational} when every component of its complement is either simply connected or an annulus containing a cusp.

A {\em measured lamination} is a geodesic lamination equipped with a transverse invariant measure.
The support of a measured lamination is a geodesic lamination having the property that the entire lamination  coincides with the union of its minimal components.
We denote the support of a measured lamination $\lambda$ by $|\lambda|$.
Conversely, a geodesic lamination with this property always supports a transverse invariant measure.
We always assume that the support of a measured lamination is the entire lamination.

For two measured laminations $\lambda$ and $\mu$, their intersection number $\iota(\lambda, \mu)$ is defined to be the integral of the product of the transverse measures of $\lambda$ and $\mu$ over the surface $S$.
In particular, when $c$ is a simple closed curve, we regard $c$ as having the unit Dirac transverse measure and define $\iota(\lambda, c)$ as such.
For two geodesic laminations $\lambda$ and $\lambda'$ on a hyperbolic surface $S$, and a point $p \in \lambda \cap \lambda'$, we can consider the  angle formed by $\lambda$ and $\lambda'$ at $p$ taking the value in $[0,\pi/2]$, which we denote by $\angle_p(\lambda, \lambda')$.
We define the angle between $\lambda$ and $\lambda'$ to be $\sup_{p \in \lambda \cap \lambda'} \angle_p(\lambda, \lambda')$ and denote it by $\angle_S(\lambda,\lambda')$.

A measured lamination $\lambda$ is said to be {\em uniquely ergodic} if the transverse measure of $\lambda$ is a unique transverse measure on its support up to scaling.
A pair of geodesic (or measured) laminations $\lambda_1$ and $\lambda_2$ is said to {\em fill up} $S$ if  every geodesic lamination $\mu$ on $S$ intersects $\lambda_1$ or $\lambda_2$ transversely.

For a  measured lamination $\lambda$ or a geodesic lamination supporting a  measured lamination on $S$, its {\em minimal supporting surface} $S(\lambda)$ is an incompressible compact subsurface containing $\lambda$ and is minimal with respect to the inclusion, which is unique up to isotopy.
%We also need a possibly larger surface, the minimal subsurface with geodesic boundary containing $\lambda$, which
%  we denote by $\bar S(\lambda)$ and call the minimal supporting {\em geodesic} surface.
%The difference between the two surfaces lies in the fact that $S(\lambda)$ may have two boundary components which are parallel to each other, in which case the annulus between them is contained in $\bar S(\lambda)$.

Geodesic laminations and measured laminations defined above depend on the hyperbolic metric given on $S$.
Still for two complete hyperbolic metrics $m, n$ on $S$, and a geodesic lamination $\lambda$ on $(S,m)$, there is a unique geodesic lamination $\lambda'$ on $(S,n)$ which is isotopic to $\lambda$.
By identifying $\lambda$ and $\lambda'$ as above, we can talk about geodesic laminations and measured laminations without specifying a hyperbolic metric.
We note the intersection number does not depend on the choice of a hyperbolic metric whereas the angle does depend on it.

Thurston proved that the space of measured laminations with the weak topology with respect to the transverse measures is homeomorphic to the Euclidean space of dimension $6g-6+2b$, where $g$ is the genus and $b$ is the number of punctures of $S$.
We denote this space by $\ml(S)$ and call it the {\em measured lamination space} of $S$.
There is a PL local chart of $\ml(S)$, which can be constructed  using train tracks as in the next section.

\subsection{Train tracks}
\label{train track}
We shall define basic terms on train tracks in this subsection.
We refer the reader to Penner-Harer \cite{PH} for a more detailed account.

A {\em train track} $\tau$ on $S$ is a $C^1$-graph (i.e.\ a graph whose edges are $C^1$-arcs and tangent to each other at vertices) embedded in $S$ whose edges are called {\em branches} and whose vertices are called {\em switches}, such that no component of $S\setminus \tau$ is a disc with one corner or an annulus with $C^1$-smooth boundary.
A {\em weight system} $\omega$ on a train track $\tau$ is a system of non-negative numbers, called {\em weights},  given on branches of $\tau$ such that at each switch the sum of the weights on the incoming branches coincides with the sum of the weights on the outgoing branches.

A geodesic lamination $\lambda$ is said to be {\em carried by} a train track $\tau$, when it can be regularly homotoped to an immersion in $\tau$.
Any geodesic lamination has a train track  carrying it.
In particular, if a measured lamination $\lambda$ is carried by a train track $\tau$, it induces a weight system on $\tau$, by defining the weight of a branch to be  the total transverse measure of the leaves lying there (after a regular homotopy).
We denote this weight system induced from $\lambda$ by $w(\lambda)$.
Conversely, for any weight system $\omega$ on a train track, we can construct a measured lamination $\lambda$ such that $w(\lambda)=\omega$.

A train track $\tau$ is said to be {\em recurrent} if it has a weight system which takes only positive values, and {\em transversely recurrent} if for each branch $b$ of $\tau$, there is a simple closed curve intersecting $\tau$ essentially (i.e. without cobounding a bigon) and transversely with non-empty intersection with $b$.
Train tracks which are both recurrent and transversely recurrent are called {\em bi-recurrent}.
Every measured lamination is carried by a bi-recurrent train track. 
For a bi-recurrent train track $\tau$, the set of measured laminations inducing weight systems with positive values on $\tau$ forms an open set in $\ml(S)$, which we denote by $U(\tau)$.
For an arational measured lamination $\lambda$, the open sets $U(\tau)$ for all bi-recurrent train tracks $\tau$ carrying $\lambda$ form a base of neighbourhoods of $\lambda$ in $\ml(S)$.

\subsection{Pleated surfaces}
Let $M$ be a hyperbolic 3-manifold, and $F$ an orientable surface, which we assume  to be either closed or the interior of a compact surface.
A {\em pleated surface} $f \colon (F, m) \to M$, where $m$ is a complete hyperbolic metric on $F$,  is a continuous map taking each cusp of $(F,m)$ to a cusp of $M$ such that for every point $x \in F$, there is at least one geodesic segment containing $x$ in its interior which is mapped isometrically to a geodesic segment in $M$ by $f$.
The set of points on $F$ at which only one direction is mapped geodesically constitutes a geodesic lamination $\mu$ on $(F,m)$.
We call  $\mu$ the {\em pleating locus} of the pleated surface $f$.
More generally, if a geodesic (or a measured) lamination $\lambda$ is mapped geodesically by a pleated surface $f$, we say that $f$ {\em realises} $\lambda$.

The boundary component of the convex core of a hyperbolic 3-manifold is an example of pleated surface.
It has moreover a special property that the surface is bent only in one direction.
The pleating locus of such a surface has a transverse measure coming from bending angles, and is called the {\em bending lamination} when it is regarded as a measured lamination.

\subsection{Ending laminations}
\label{EL}
By Bonahon's tameness theorem \cite{Bon}, it is known that 
for any faithful discrete representation $\phi \colon \pi_1(S) \to \pslc$, there is an orientation-preserving homeomorphism  $\Phi \colon S \times (0,1) \to \hyperbolic^3/\phi(\pi_1(S))$ which induces $\phi$ between their fundamental groups.

For a hyperbolic 3-manifold $M=\hyperbolic^3/\phi(\pi_1(S))$, its {\em non-cuspidal part}, denoted by $M_0$, is the complement of $\epsilon$-thin cusp neighbourhoods for some fixed positive number $\epsilon$ smaller than the three-dimensional Margulis constant.
The boundary of $M_0$ consists of incompressible open annuli.
(In the case of general Kleinian groups, incompressible tori may appear.
We do not have such components since we only deal with Kleinian surface groups.)
By the relative core theorem (\cite{Sc, Mc}), there is a compact submanifold $C$ of $M_0$, called a {\em relative compact core} of $M_0$,  such that the inclusion is a homotopy equivalence and $C \cap \partial M_0$ is the union of core annuli of the components of $\partial M_0$.
A core curve of each component of $C \cap \partial M_0$ represents a generator of a maximal parabolic subgroup of $\phi(\pi_1(S))$.
We call these curves {\em parabolic curves}.
A relative compact core $C$ of $M_0$ is always homeomorphic to $S \times [0,1]$ in our setting.
We fix orientations on $S$ and $\hyperbolic^3$, and assume the identification of $S \times [0,1]$ with $C$ to preserve the orientations.
%Each component of $C \cap \partial M_0$ is an open annulus whose core curve represents a generator of a maximal parabolic subgroup of $\phi(\pi_1(S))$.
%We call these curves {\em parabolic curves}.

An {\em end} of $M_0$ is an inverse limit (with respect to the inclusion) of complementary components of compact sets in $M_0$.
Each component $U$ of $M_0 \setminus C$ contains a unique end $e$ of $M_0$, and also its closure contains a unique component $\Sigma$ of $\Fr_{M_0} C$, where $\Fr_{M_0}$ denotes the frontier as a subspace of $M_0$.
In this situation, we say that $\Sigma$ {\em faces} the end $e$.
The end $e$ is said to be {\em geometrically finite} when it has a neighbourhood disjoint from any closed geodesic, and otherwise {\em geometrically infinite}.
If $e$ is geometrically finite, there is a boundary component $F$ of the convex core $C(M)$ such that $F \cap M_0$ is isotopic to $\Sigma$.
This component $F$ in turn corresponds to a component of $\Omega_G/G$ which is regarded as lying at infinity.

When $e$ is geometrically infinite, it was proved in \cite{Bon} that there is a sequence of simple closed curves $c_i$ on $\Sigma$ which are homotopic in $U \cup \Sigma$ to  closed geodesics $c_i^*$ tending to the end.
Such an end is called {\em simply degenerate}.
Regarding $c_i$ as a geodesic lamination on $\Sigma$, after fixing any hyperbolic metric on $\Sigma$, we consider the Hausdorff limit $c_\infty$ of $c_i$, which is a geodesic lamination.
It was shown by Thurston \cite{ThL} and Bonahon \cite{Bon} that $c_\infty$ has only one minimal component $\lambda$, which  is called the {\em ending lamination} of $e$, and that $S(\lambda)=\Sigma$.
The geodesic lamination $\lambda$  is the support of a measured lamination which is a limit of $\{r_i c_i\}$ in the space of measured laminations, where $r_i$ is a positive scalar.

The notion of ending lamination was first introduced by Thurston using pleated surfaces as follows.
Let $\Sigma$ be a subsurface of $S$ as above.
If the end $e$ facing $\Sigma$ is geometrically infinite, there is a sequence of pleated surfaces $\{f_i\}$ homotopic to the inclusion of $\Sigma$ which tends to $e$.
For instance, in the setting of the preceding paragraph,   pleated surfaces realising the simple closed curves $c_i$ are  such  pleated surfaces.
Thurston considered the Hausdorff limit of the geodesic laminations realised by such pleated surfaces, and proved that the limit has only one minimal component, which is defined to be the ending lamination of $e$.
He also showed that the ending lamination thus defined does not depend on the choice of pleated surfaces.

Recall that the relative compact core $C$ is identified with $S \times [0,1]$.
When a parabolic curve lies on $S_+=S\times \{1\}$ (resp.\ $S_-=S \times \{0\}$), we call it an {\em upper} (resp.\ a {\em lower}) parabolic curve.
When the end $e$ is above $C$, i.e., when $\Sigma$ lies on $S \times \{1\}$ (resp. $S \times \{0\}$), we say that the ending lamination $\lambda$ is an upper (resp. a lower) ending lamination.
It was also proved in \cite{ThL} and \cite{Bon} that for each upper (resp. lower)  ending lamination $\lambda$ of $\hyperbolic^3/\phi(\pi_1(S))$, each boundary component of $S(\lambda)$ is an upper (resp. lower) parabolic curve.
%
%
%We regard the ending laminations  of the upper ends of $(\hyperbolic^3/\phi(\pi_1(S))_0$ and the upper parabolic curves as lying on $S_+=S\times \{1\}$ and the ending laminations of lower ends and lower parabolic curves as lying on $S_- =S\times \{0\}$.
We call the union of the parabolic curve and the ending laminations  regarded as lying on $S_-\sqcup S_+$ the {\em qi(quasi-isometric)-end invariant} of $\phi$ (or $\hyperbolic^3/\phi(\pi_1(S))$).
In particular, the union of those lying on $S_+$ (resp. $S_-$) is  called the upper (resp. lower) qi-end invariant.

\subsection{Deformation spaces}
\label{DS}
The space of faithful discrete representations of $\pi_1(S)$ into $\pslc$ modulo conjugacy is denoted by $\AH(S)$.
We endow $\AH(S)$ with the topology induced from the representation space.
Although each element of $\AH(S)$ is a conjugacy class of representations,
by abusing notation, we denote it by its representative.

The interior of $\AH(S)$ is known to be the quasi-Fuchsian space $\QF(S)$.
When $\phi \in \AH(S)$ is quasi-Fuchsian, letting $G$ be $\phi(\pi_1(S))$, the conformal structure on $\Omega_G/G$ induces a marked conformal structures at infinity, on $S \times \{0\}$ and $S \times \{1\}$.
We note that the identification of $S \times \{0\}$ to a component of $\Omega_G/G$ is orientation-preserving, but that of $S \times \{1\}$ is orientation-reversing.
We use the symbol $\teich(\bar S)$ to denote the Teichm\"{u}ller space of $S$ with its orientation reversed.
Then the conformal structure on $\Omega_G/G$ determines a point in $\teich(S) \times \teich(\bar S)$.
Bers showed that this identification of the conformal structure on $\Omega_G/G$ and a point in  $\teich(S) \times \teich(\bar S)$  gives a parametrisation of $\QF(S)$, which we denote by $q \colon \teich(S) \times \teich(\bar S) \to \QF(S)$.

For a general point $\phi \in \AH(S)$ and $G=\phi(\pi_1(S))$, let $\lambda_-$ and $\lambda_+$ be the lower and upper qi-end invariants of $\hyperbolic^3/\phi(\pi_1(S))$.
We recall that $\lambda_-$ (resp.\ $\lambda_+$) has the property that for any component $\lambda$ of $\lambda_-$ (resp.\ $\lambda_+$), every boundary component of $S(\lambda)$ is contained in $\lambda_-$ (resp.\ $\lambda_+$).
The quotient of the region of discontinuity $\Omega_G/G$ is identified with the disjoint union of $S_- \setminus  S(\lambda_-)$ and $S_+\setminus  S(\lambda_+)$, which we denote by $\Sigma_-$ and $\Sigma_+$.
The ending lamination theorem, proved by Minsky and Brock-Canary-Minsky \cite{Mi,BCM}, shows that any Kleinian group in $\AH(S)$ having $\lambda_-$ and $\lambda_+$ as lower and upper qi-end invariants is a quasi-conformal deformation of $G$.
Therefore, we denote the quasi-conformal deformation space of $G$ by $\QH_{\lambda_-, \lambda_+}$.
The theory of Bers-Kra-Maskit-Marden-Sullivan shows that the conformal structures at infinity give a parametrisation $q \colon \teich(\Sigma_-) \times \teich(\Sigma_+) \to \QH_{\lambda_-, \lambda_+}$.

\subsection{Geometric limits}
For a sequence of Kleinian groups $\{G_i\}$, we say that $\{G_i\}$ converges to a Kleinian group $\Gamma$ {\em geometrically} if (i) every element $\gamma \in \Gamma$ is a limit of some sequence $\{g_i \in G_i\}$, and (ii) for every convergent subsequence $\{g_{i_j} \in G_{i_j}\}$, its limit lies in $\Gamma$.
The geometric convergence is equivalent to the pointed Gromov-Hausdorff convergence of the corresponding hyperbolic 3-manifolds: fixing a basepoint $x \in \hyperbolic^3$ and letting $x_i$ and $x_\infty$ be the projections of $x$ to $\hyperbolic^3/G_i$ and $\hyperbolic^3/\Gamma$ respectively, the sequence of pointed hyperbolic 3-manifolds $\{(\hyperbolic^3/G_i, x_i)\}$ converges to $(\hyperbolic^3/\Gamma, x_\infty)$ in the sense of Gromov-Hausdorff if and only if $\{G_i\}$ converges to $\Gamma$ geometrically.
Due to this fact, we also refer to Gromov-Hausdorff limits as geometric limits.
The compactness of Gromov-Hausdorff topology shows that every sequence of non-elementary Kleinian groups has a geometric limit after passing to a subsequence.
We recall that, by definition, if $\{(\hyperbolic^3/G_i, x_i)\}$ converges to $(\hyperbolic^3/\Gamma, x_\infty)$ in the sense of Gromov-Hausdorff, then there is a $K_i$-bi-Lipschitz diffeomorphism, which is called an {\em approximate isometry} between the $R_i$-ball around $x_i$ and the $K_iR_i$-ball around $x_\infty$, with $K_i \longrightarrow 1$ and $R_i\longrightarrow \infty$.

Suppose that a sequence $\{\phi_i\}$ in $\AH(S)$ converges to $\psi \in \AH(S)$.
In this situation, we always assume that we take representatives so that $\{\phi_i\}$ converges to $\psi$ as genuine representations from $\pi_1(S)$ into $\pslc$.
Then, passing to a subsequence,  $\{\phi_i(\pi_1(S))\}$ converges to some Kleinian group $\Gamma$.
From the definition of geometric limits, it is easy to see that $\Gamma$ contains $\psi(\pi_1(S))$.
If $\Gamma=\psi(\pi_1(S))$, we say that $\{\phi_i\}$ converges to $\psi$ {\em strongly}.

In Ohshika-Soma \cite{OS} a classification of geometric limits of Kleinian surface groups was given.
An alternative description was also given in \cite{OhT}.
We shall now review some of the results there which will be used in the proof of the main theorem, in particular, in \cref{b}.

Let $\{\phi_i\}$ be a sequence in $\AH(S)$, and suppose that $\{\phi_i(\pi_1(S))\}$ converges geometrically to a Kleinian group $\Gamma$.
The following is a paraphrase of a part of Theorem A in \cite{OS}.

\begin{theorem}
\label{OS}
The non-cuspidal part $(\hyperbolic^3/\Gamma)_0$ is topologically embedded  in $S \times (0,1)$ in such a way that the following hold.
\begin{enumerate}[(a)]
\item
Every end of $(\hyperbolic^3/\Gamma)_0$ is mapped to a horizontal surface $\Sigma \times \{t\}$, where $\Sigma$ is an incompressible subsurface of $S$ and $t$ lies in $[0,1]$.
\item Every geometrically finite end of $(\hyperbolic^3/\Gamma)_0$ lies on either $S \times \{0\}$ or  $S \times \{1\}$.
\item Each boundary component of the convex core of $(\hyperbolic^3/\Gamma)_0$ is a Gromov-Hausdorff limit of  boundary components of convex cores of $(\hyperbolic^3/\phi_i(\pi_1(S))_0$ with some base points.
\item Every geometrically infinite end is either simply degenerate or an accumulation set of countably many torus cusps or simply degenerate ends or both.
\end{enumerate}
\end{theorem}

Identifying $(\hyperbolic^3/\Gamma)_0$ with the image of its embedding in $S \times (0,1)$ as above, we can talk about the horizontal direction and the vertical direction in $(\hyperbolic^3/\Gamma)_0$.

The following lemma, which is \cite[Lemma 4.13]{OhT} and also can be found in \cite[\S3]{BBCL}, will be used in \cref{b}.
\begin{lemma}
\label{alg loc}
In the setting as above, suppose moreover that $\{\phi_i\}$ converges to $\psi \in \AH(S)$.
Then the image of the inclusion of $\psi(\pi_1(S))$ into $(\hyperbolic^3/\Gamma)_0$ is represented by an immersion $f_\infty \colon S \to (\hyperbolic^3/\Gamma)_0$ which is horizontal except for disjoint annuli in $S$ whose images wrap around torus cusps.
\end{lemma}

An immersion $f_\infty$ as above is called an {\em algebraic locus}.
We note that for an approximate isometry $\rho_i$ from $\hyperbolic^3/\phi_i(\pi_1(S))$ to $\hyperbolic^3/\Gamma$, the composition $\rho_i^{-1} \circ f_\infty$ induces the same isomorphism as $\phi_i$ between the fundamental groups for sufficiently large $i$.

By \cref{OS,alg loc}, we can show the following.
\begin{corollary}
\label{strong}
%Let $\{\phi_i\}$ be a sequence of quasi-conformal deformations of a Kleinian surface group $\phi \in \AH(S)$.
Suppose that $\{\phi_i \}$ in $\AH(S)$ converges to $\psi \in \AH(S)$.
If $(\hyperbolic^3/\Gamma)_0$ does not have a torus boundary component and every geometrically infinite end is homotopic into an algebraic locus, then $\{\phi_i\}$ converges to $\psi$ strongly.
\end{corollary}
\begin{proof}
Since we assumed that $(\hyperbolic^3/\Gamma)_0$ does not have a torus boundary component, the algebraic locus $f_\infty$ in \cref{alg loc} cannot wrap around a boundary component, and hence is a horizontal surface.
By assumption, every geometrically infinite end can be lifted to the algebraic limit, and hence simply degenerate.
This implies that in the embedded image of $(\hyperbolic^3/\Gamma)_0$ in $S \times (0,1)$, there is no other end on the side farther from $f_\infty(S)$ of each geometrically infinite end.
Therefore, we can isotope the embedding so that every geometrically infinite end  lies on $S \times \{0\} \cup S \times \{1\}$.
%simply degenerate endもtopとbottomにおく．

By \cref{OS}-(a), every geometrically finite also lies on $\Sigma \times \{0,1\}$.
Each boundary component of $(\hyperbolic^3/\Gamma)_0$, which is an open annulus, has both ends on the same level, either on $S\times \{0\}$ or $S \times \{1\}$, for it cannot pass through the algebraic locus.
This means that $(\hyperbolic^3/\Gamma)_0$ coincides with the complement of finitely many half solid tori lying on a neighbourhood of $S\times \{0\}$ or $S \times \{1\}$.
Therefore it is homeomorphic to $S \times (0,1)$, and hence $\psi(\pi_1(S))=\Gamma$.
\end{proof}

As explained in  \cite{OS, OhT}, if a sequence $\{\phi_i\}$ in $\AH(S)$ converges geometrically to a Kleinian group $\Gamma$, a geometric limit of (uniform) bi-Lipschitz model manifolds of $(\hyperbolic^3/\phi_i(S))_0$ (due to Minsky \cite{Mi}) serves as a bi-Lipschitz model manifold of $(\hyperbolic^3/\Gamma)_0$.
Suppose that $\hyperbolic^3/\phi_i(\pi_1(S))$ has a lower qi-end invariant $\lambda_-$ and an upper qi-end invariant $\lambda_+$.
Let $\Sigma_-$ and $\Sigma_+$ be $S\setminus  S(\lambda_-)$ and $S \setminus S(\lambda_+)$ respectively, and 
$\mathbf m_i^-$ and $\mathbf m_i^+$  the  structures at infinity of $\hyperbolic^3/\phi_i(\pi_1(S))$ on $\Sigma_-$ and $\Sigma_+$ respectively. 
Let $P_i^-$ and $P_i^+$ be shortest pants decompositions (with respect to the hyperbolic length) of $(\Sigma_-, \mathbf m_i^-)$ and $(\Sigma_+, \mathbf m_i^+)$ respectively.
By adding a shortest transversal simple closed curve to each component of $P_i^-$  (resp.$P_i^+$) disjoint from all the other components, we get shortest markings $M_i^-$  (resp. $M_i^+$).
Minsky's model manifold is constructed from the hierarchy of tight geodesics connecting the generalised markings $M_i^-\cup \lambda_-$ and $M_i^+ \cup \lambda_+$.

In this setting, we have the following lemmas.

\begin{lemma}
\label{torus}
Suppose that $\{\phi_i(\pi_1(S))\}$ converges geometrically to $\Gamma$.
Let $l$ be a non-contractible simple closed curve on $S$.
We assume that the translation length of $\phi_i(l)$ is bounded from above as $i \longrightarrow \infty$.
Let $\epsilon$ be a positive constant less than the three-dimensional Margulis constant.

\begin{enumerate}
\item
If the $\epsilon$-Margulis tube around the closed geodesic representing $\phi_i(l)$ converges geometrically to a torus cusp neighbourhood in $\hyperbolic^3/\Gamma$ as $i \longrightarrow \infty$, then both $\len_{\mathbf m_i^+}(l)$ and $\len_{\mathbf m_i^-}(l)$  (if one or both of them are defined) are bounded from below by a positive constant, and $d_{A(l)}(M_i^- \cup \lambda_-, M_i^+ \cup \lambda_+)$ goes to $\infty$.
Here $A(l)$ denotes an annular neighbourhood of $l$ and $d_{A(l)}$ the distance in the curve complex of $A(l)$.
Conversely if both $\len_{\mathbf m_i^+}(l)$ and $\len_{\mathbf m_i^-}(l)$ are bounded from below by a positive constant and $d_{A(l)}(M_i^-\cup \lambda_-, M_i^+ \cup \lambda_+)$ goes to $\infty$, then the $\epsilon$-Margulis tube around the closed geodesic representing $\phi_i(l)$ converges geometrically to either a torus cusp neighbourhood or a $\integers$-cusp attached to a geometrically infinite end.

\item
Suppose moreover that $\{\phi_i\}$ converges in $\AH(S)$, and let $k$ be a simple closed curve intersecting $l$ essentially.
If the $\epsilon$-Margulis tube around the closed geodesic representing $\phi_i(l)$ converges geometrically to a torus cusp lying above (resp.\ below) an algebraic locus as $i \longrightarrow \infty$, then $d_{A(l)}(k, M_i^+ \cup \lambda_+)$ (resp.\ $d_{A(l)}(k, M_i^-\cup \lambda_-)$) goes to $\infty$.
In the case when an algebraic locus wraps around a torus cusp, we regard the cusp as lying both above and below the locus.
Conversely, if $d_{A(l)}(k, M_i^+\cup \lambda_+)$ (resp.\ $d_{A(l)}(k, M_i^-\cup \lambda_-)$) goes to $\infty$, then the $\epsilon$-Margulis tube around the closed geodesic representing $\phi_i(l)$ converges geometrically to either a torus cusp neighbourhood or a $\integers$-cusp neighbourhood attached to a geometrically infinite end.
\end{enumerate}
\end{lemma}
\begin{proof}
Recall that the model manifold of $(\hyperbolic^3/\Gamma)_0$ can be taken to be a geometric limit of models constructed from hierarchies of tight geodesics.
If $\len_{\mathbf m_i^-}(l)$ or $\len_{\mathbf m_i^+}(l)$ goes to $0$ after passing to a subsequence, then the corresponding geometrically finite block of the model manifold of $(\hyperbolic^3/\phi_i(\pi_1(S)))_0$ splits along a simple closed curve representing $l$ as $i \longrightarrow\infty$.
This shows that the geometric limit of the Margulis tubes around the closed geodesic representing $\phi_i(l)$ must be a $\integers$-cusp in $\hyperbolic^3/\Gamma$ in this case.

Now suppose that $\len_{{\mathbf m}_i^-}(l)$ and $\len_{\mathbf m_i^+}(l)$ are bounded from below by a positive constant if one or both of them are defined.
Then by \cite[Lemma 9.4]{Mi}, we see that the Margulis tubes around the closed geodesic corresponding to $\phi_i(l)$ converge geometrically to a torus cusp neighbourhood if and only if $d_{A(l)}(M_i^-\cup \lambda_-, M_i^+ \cup \lambda_+)$ goes to $\infty$.
Thus we are done for the part (1).

A proof of the part (2) can be found in \cite[Proof of Theorem 5.2]{OhT}.
We summarise its argument here.
Consider the situation where the Margulis tube $V_i$ around the closed geodesic corresponding to $\phi_i(l)$ converges to a torus cusp neighbourhood $V_\infty$.
First suppose that the algebraic locus does not wrap around $V_\infty$ and that  $V_\infty$ lies above an algebraic locus.
Recall that Minsky's model manifold is constructed from a hierarchy $H_i$ of tight geodesics in the curve complex of (subsurfaces of) $S$.
If the Margulis tube $V_i$ converges to a torus cusp neighbourhood, then the annular neighbourhood $A(l)$ supports a geodesic $g_i$ in $H_i$ whose length goes to $\infty$.
The last vertex of $g_i$ is within bounded distance from $\pi_{A(l)}(M_i^+\cup \lambda_+)$ as $i \longrightarrow \infty$, where $\pi_{A(l)}$ denotes the projection between the curve complexes $\cc(S)$ and $\cc(A(l))$ induced by restricting curves to $A(l)$.
On the other hand, since the torus cusp lies above an algebraic locus, the initial vertex of $g_i$ is within bounded distance from $\pi_{A(l)}(k)$ as $i \longrightarrow \infty$.
Thus, we have $d_{A(l)}(k, M_i^+ \cup \lambda_+)\longrightarrow \infty$.
The same argument works also in the case when $V_\infty$ lies below an algebraic locus.
(See \cite[Claim 5.3]{OhT} for more details.)

Next suppose that an algebraic locus wraps $n$-times around $V_\infty$ with $n \neq 0$.
Then there is a sequence of integers $r(i)$ whose absolute values go to $\infty$ such that the initial vertex of $h_i$ is within bounded distance from $\pi_{A(l)}(\tau_l^{nr(i)}(k))$ and the last vertex is within bounded distance from $\pi_{A(l)}(\tau_l^{(n+1)r(i)}(k))$, where $\tau_l$ denotes the  Dehn twist around $l$.
The initial vertex of $h_i$ is also within bounded distance from $\pi_{A(l)}(M_i^-\cup \lambda_-)$, and 
the last vertex of $h_i$ is also within bounded distance from $\pi_{A(l)}(M_i^+ \cup \lambda_+)$.
Therefore both $d_{A(l)}(k, M_i^-\cup \lambda_-)$ and $d_{A(l)}(k, M_i^+ \cup \lambda_+)$ go to $\infty$ in this case.

Finally, suppose that $d_{A(l)}(k, M_i^-\cup \lambda_-)$ goes to $\infty$.
Then by \cite[Lemma 9.4]{Mi} again, we see that the boundary of the Margulis tube $\partial V_i$ either converges to the boundary of a torus cusp neighbourhood or diverges to give rise to an open annulus whose end is attached to a lower end of $(\hyperbolic^3/\Gamma)_0$.
Since $d_{A(l)}(k,M_i^-\cup \lambda_-) \longrightarrow \infty$, the curve $l$ cannot be contained in the shortest pants decomposition $M_i$.
Therefore, the upper geometrically finite block of the model manifold of $(\hyperbolic^3/\phi_i(\pi_1(S)))_0$ cannot split along $l$, and hence neither end of the geometric limit of $\partial V_i$ cannot tend to $S\times \{0\}$.
%Since the geometric convergence of the model of $(\hyperbolic^3/\phi_i(\pi_1(S))_0$ to $(\hyperbolic^3/\Gamma)_0$ preserve the horizontal levels of their embeddings into $S \times (0,1)$, the geometric limit of $V_i$ cannot have an end lying on $S \times \{0\}$ or $S\times \{1\}$.
It follows that the end of $(\hyperbolic^3/\Gamma)_0$ to which an end of the geometric limit of $\partial V_i$ tends must be geometrically infinite by \cref{OS}-(b).
Thus we have shown that the second part of (2) holds.
\end{proof}

\begin{lemma}
\label{simply deg}
Suppose  that $\{\phi_i\}$ converges in $\AH(S)$ to $\psi$.
Let $\Sigma$ be an open incompressible subsurface of $S$, with negative Euler characteristic.
%Fix some simple closed curve $c_0$ in $\Sigma$.
Then,  $\{P_i^+|\Sigma\}$ (resp.\ $\{P_i^-|\Sigma\}$) converges (in the Hausdorff topology) to a geodesic lamination containing a minimal component $\lambda$ whose minimal supporting surface is $\Sigma$ if and only if there is an upper (resp. a lower) simply degenerate end of $(\hyperbolic^3/\psi(\pi_1(S)))_0$ whose ending lamination is $\lambda$.
\end{lemma}
\begin{proof}
This is contained in \cite[Proposition 4.18, Theorem 5.2]{OhT} and their generalisations explained in \cite[\S10]{OhT}, or alternatively in Brock-Bromberg-Canary-Lecuire \cite[Theorem 1.2]{BBCL}.
\end{proof}

\section{Main theorems}
The purpose of this paper is to prove the following theorem, which is a generalisation of the existence part of the theorem of Bonahon-Otal \cite{BO} to general, possibly geometrically infinite Kleinian surface groups.
We do not have the uniqueness part of Bohaon-Otal's theorem (in the case when the bending laminations are multi-curves), but instead have compactness for the set of Kleinian groups realising given data of qi-end invariants and bending laminations.

\begin{theorem}
\label{main}
Let $S$ be a closed oriented surface of genus greater than $1$, and $S_-, S_+$ two copies of $S$, where $S_+$ has the same orientation as $S$ whereas $S_-$ has the opposite orientation.
\begin{itemize}
\item
Let $\lambda_-$ and $\lambda_+$ be  (possibly empty) geodesic laminations without non-compact isolated leaves on $S_-$ and $S_+$  respectively, such that for every component $\lambda$ of $\lambda_-$ (resp.\ $\lambda_+$), each boundary component of the minimal supporting surface $S(\lambda)$ is isotopic to a closed geodesic contained in $\lambda_-$ (resp.\ $\lambda_+$).

\item
Let $\mu_-$ and $\mu_+$ be measured laminations  on  $S_-$ and $S_+$ such that
\begin{enumerate}[{\rm (a)}]
\item $\lambda_- \cap \mu_-=\emptyset$ and $\lambda_+ \cap \mu_+=\emptyset$;
\item neither $\mu_-$  nor $\mu_+$ contains a compact leaf of weight larger than or equal to $\pi$; and
\item  $\lambda_-\sqcup \mu_-$ and $ \lambda_+ \sqcup \mu_+$ fill up $S$ if we identify both $S_-$ and $S_+$ with $S$.
\end{enumerate}
\end{itemize}
Then the following hold.
\begin{enumerate}[{\rm (1)}]
\item 
\label{end and bend}
There is $\phi \in \AH(S)$ such that
\begin{enumerate}[{\rm (i)}]
\item
the hyperbolic 3-manifold $\hyperbolic^3/\phi(\pi_1(S))$  has $\lambda_-$  as its lower qi-end invariant and $\lambda_+$ as its upper qi-end invariant, and
\item the hyperbolic 3-manifold $\hyperbolic^3/\phi(\pi_1(S))$ realises  $\mu_-$ and $\mu_+$ as the bending laminations on the lower and the upper boundaries respectively of its convex core $C(\hyperbolic^3/\phi(\pi_1(S))$.
\end{enumerate}
\item The set of all $\rho \in \AH(S)$ satisfying the condition {\rm (\ref{end and bend})} is a compact subset of $\QH_{\lambda_-, \lambda_+}$. 
\end{enumerate}
\end{theorem}

\begin{remark}
In the theorem above, the existence of $\mu_-, \mu_+$ satisfying (c) imposes on $\lambda_-, \lambda_+$ the condition that they share no minimal component.
\end{remark}

This theorem is an immediate consequence of  following \cref{proper degree 1}, which states something a bit stronger.

Let $\lambda_-$ and $\lambda_+$ be geodesic laminations on $S_-$ and $S_+$ as in \cref{main}, which share no minimal component.
%Let $\bar S(\lambda_-)$ and $\bar S(\lambda_+)$ be the minimal geodesic supporting surfaces of $\lambda_-$ and $\lambda_+$ respectively, as defined in \cref{mss}.
Let $\Sigma_-$ and $\Sigma_+$ be the complements $S_- \setminus S(\lambda_-)$ and $S_+ \setminus  S(\lambda_+)$ respectively.
As we explained in \cref{DS}, we have a parametrisation $q \colon \teich(\Sigma_-) \times \teich(\Sigma_+) \to \QH_{\lambda_-, \lambda_+}$.
Let $b \colon \QH_{\lambda_-, \lambda_+} \to \ml(\Sigma_-) \times \ml(\Sigma_+)$  be the map taking $\rho\in \QH_{\lambda_-, \lambda_+}$ to the bending lamination of $C(\hyperbolic^3/\rho(\pi_1(S)))$. 

\begin{theorem}
\label{proper degree 1}
Let $D$ be the subset of $\ml(\Sigma_-) \times \ml(\Sigma_+)$ consisting of measured laminations which do not satisfy at least one of the conditions (b) and (c) in \cref{main}.
Then the map $b \circ q \colon \teich(\Sigma_-) \times \teich(\Sigma_+) \to \ml(\Sigma_-) \times \ml(\Sigma_+) \setminus D$ is a proper, degree-1 map.
\end{theorem}

\section{Bending laminations}
%\subsection{Lengths and intersection numbers of curves}
We present some lemmas which will be useful for the proofs of our main theorems.
\begin{lemma}
\label{bound int lem}
For every $K >0$, there is a positive constant $L(K)$ depending only on $K$ which goes to $0$ as $K \longrightarrow 0$ with the following property.
For every $\phi \in \AH(S)$, letting $M=\hyperbolic^3/\phi(\pi_1(S))$, if $\Sigma$ is a boundary component of $C(M)$ with bending lamination $\lambda$ and c is
 a simple closed curve on  $\Sigma$ with length less than $K$, then  $\iota(c, \lambda)<L(K)$.
\end{lemma}
\begin{proof}
This is a direct consequence of the boundedness of the average bending measure proved by Bridgeman \cite{Bri}.
%Suppose that such a bound does not exist.
%Then there is a sequence of $\{\rho_i\}$ in $\AH(S)$ such that  for a boundary component $\Sigma_i$ of the convex core of $M_i=\hyperbolic^3/\rho_i(\pi_1(M))$ with bending lamination $\lambda_i$ has a simple closed curve $c_i$ of length bounded by $K$ but $i(c_i, \lambda_i) \to \infty$.
%Homotope $c_i$ to a closed geodesic $c^*_i$ with respect to the induced hyperbolic metric on $\Sigma_i$.
%Putting a basepoint $x_i$ on $c_i$, we can consider a geometric limit 
%The first statement can be easily seen by considering a geometric limit 
\end{proof}

As a consequence of this lemma, we have the following.

\begin{corollary}
\label{shortest int}
Consider the situation in \cref{main}, and let $\Sigma$ be a component of $\Sigma_- \sqcup \Sigma_+$.
Let $\{\mathbf g_i\}$ be a sequence in $\teich(\Sigma_-) \times \teich(\Sigma_+)$.
Let $\gamma_i$ be a shortest pants decomposition of $\mathbf g_i|\Sigma$.
Then the sequence $\{\iota(b\circ q(\mathbf g_i), \gamma_i)\}$ is bounded.
\end{corollary}
\begin{proof}
By Bers's lemma, there is a constant $C$ depending only on $S$ such that each component of the shortest pants decomposition of $(\Sigma, \mathbf g_i)$ has length less than $C$.
Therefore, by \cref{bound int lem}, $\iota(b\circ q(\mathbf g_i), \gamma_i)$ is bounded as $i \longrightarrow \infty$.
\end{proof}

%The next lemma can be easily proved by an argument of constructing ruled annuli.
\begin{lemma}
\label{homotopic geodesic}
For any $\epsilon >0$, there is a  positive constant $\delta>0$ depending only on $\epsilon$ satisfying the following.

Let $\{\phi_i\}$ be a sequence in $\QH_{\lambda_-, \lambda_+}$, and  $\Sigma$ a component of $\Sigma_- \sqcup \Sigma_+$.
Let $\mu_i$ be the restriction of the bending lamination of the convex core of $\hyperbolic^3/\phi_i(\pi_1(S))$ to (the component corresponding to) $\Sigma$.
We denote the hyperbolic metric on $\Sigma$ as a boundary component of the convex core $C(\hyperbolic^3/\phi_i(\pi_1(M)))$ by $m_i'$.
 
Suppose  
\begin{enumerate}[(i)]
\item that $\{\mu_i\}$ converges to  a measured lamination $\mu_\infty$,
\item
that  $\mu_\infty$ is decomposed into disjoint (possibly empty) sublaminations $\mu_\infty^1$ and $\mu_\infty^2$,
\item and that $c$ is a simple closed curve on $\Sigma$ such that $\iota(c, \mu_\infty^1) < \delta$ and $\angle_{(\Sigma,  m_i')}(c, \mu_\infty^2)< \delta$ for large $i$.
\end{enumerate}

Let $c_i$ be the closed geodesic on $\Sigma$ with respect to $m_i'$ freely homotopic to $c$, and $c_i^*$ the closed geodesic in $\hyperbolic^3/\phi_i(\pi_1(S))$ freely homotopic to $\phi_i(c)$.

Then we have $$1 \leq \frac{\len_{(\Sigma, m_i')}(c_i)}{\len_{\hyperbolic^3/\phi_i(\pi_1(S))}(c_i^*)} \leq 1+\epsilon.$$
\end{lemma}
\begin{proof}
Since the length of $c_i^*$ is less than or equal to that of $c_i$, the first inequality obviously holds.

Since the Hausdorff limit of the support $|\mu_i|$ is a geodesic lamination containing $\mu_\infty=\mu_\infty^1 \sqcup \mu_\infty^2$, for sufficiently large $i$, we can decompose $c_i$ into two parts $c_i^2$ and $c_i^1$ such that $\angle_{(\Sigma,m_i')} (c_i^2, \mu_i) < 2\delta$ and $\iota(c_i^1, \mu_i)< 2\delta$.
By applying  \cite[Proposition 4.1]{Baba2} for $c_i^2$ and  \cite[Corollary 4.6]{EMM}  for $c_i^1$, we obtain the second inequality.
\end{proof}

\section{properness}
\label{properness}
In this section, we shall prove the properness of $b\circ q$ in \cref{proper degree 1}.
The argument is by contradiction.
Suppose that $b \circ q$ is not proper.
Then, there is a  sequence $\{(\mathbf m_i^-,\mathbf m_i^+)\} \subset \teich(\Sigma_-)\times \teich(\Sigma_+)$ without a convergent subsequence such that $\{b\circ q(\mathbf m_i^-,\mathbf m_i^+)\}$ converges in $\ml(\Sigma_-) \times \ml(\Sigma_+) \setminus D$ to a measured lamination $\nu$ on $\Sigma_- \sqcup \Sigma_+$.

We are going to analyse a geometric limit of $\{\hyperbolic^3/q(\mathbf m_i^-,\mathbf m_i^+)\}$, making use of results in Ohshika-Soma \cite{OS} and Ohshika \cite{OhT}.
Let $\phi_i \colon \pi_1(S) \to \pslc$ be a representative of $q(\mathbf m_i^-,\mathbf m_i^+)$ for each $i$.
We take the $\phi_i$ to converge to some $\psi \in \AH(S)$ if $\{q(\mathbf m_i^-,\mathbf m_i^+)\}$ converges in $\AH(S)$.

%As was shown in \cite{OhT}, there is an
%We take an efficient pleated surface $f_i \colon S \to \hyperbolic^3/\phi_i(\pi_1(S))$ inducing $\phi_i$ between the fundamental groups.
%Put a basepoint $x_i$ on $f_i(S)$, and consider a geometric limit $M_\infty$ of $(\hyperbolic^3/\phi_i(\pi_1(S)), x_i)$.
%Since $f_i$ is efficient, if $\{\phi_i\}$ converges to $\psi \colon \pi_1(S)$, then $\psi(\pi_1(S))$ is contained in $G_\infty$.

We divide our argument into the following  three cases:
\begin{enumerate}[(a)]
\item The case when $\{q(\mathbf m_i^-,\mathbf m_i^+)\}$ converges in $\AH(S)$ strongly after passing to a subsequence.

\item
The case when  $\{q(\mathbf m_i^-,\mathbf m_i^+)\}$ converges in $\AH(S)$ after passing to a subsequence, but not strongly.
\item
The case when $\{q(\mathbf m_i^-,\mathbf m_i^+)\}$ diverges in $\AH(S)$ (even after passing to a subsequence).

\end{enumerate}
In the cases (a) and (b), we put a basepoint $\tilde x$ in $\hyperbolic^3$, and by projecting it to $\hyperbolic^3/\phi_i(\pi_1(S))$, we get a basepoint $x_i$.
Taking the Gromov-Hausdorff limit of $(\hyperbolic^3/\phi_i(\pi_1(S)), x_i)$, we obtain a geometric limit $(M_\infty, x_\infty)$ which is covered by the algebraic limit $\hyperbolic^3/\psi(\pi_1(S))$.

Now we start to show  that in each of the three cases, we get a contradiction.

\subsection{Case (a)}
\label{a}
%Suppose that  $\{\phi_i\}$  converges strongly to a representation $\psi \colon \pi_1(S) \to \pslc$.
%Then $(\hyperbolic^3/\phi_i(\pi_1(S)), x_i)$ converges geometrically to $((\hyperbolic^3/\psi(\pi_1(S)), x_\infty)$, where $x_i$ and $x_\infty$ are projections of a fixed base-point $\tilde x$ in $\hyperbolic^3$.
In this case, $M_\infty=\hyperbolic^3/\psi(\pi_1(S))$.
There are three possibilities for  $\psi$:
\begin{enumerate}[{(a)}-(1)]
\item The case when $\psi$ is also contained in $\qsp$.
\item The case when $(\hyperbolic^3/\psi(\pi_1(S))_0$ has a \lq new simply  degenerate end' not corresponding to a simply degenerate end of $(\hyperbolic^3/\phi_i(\pi_1(S)))_0$.
\item The case when $\hyperbolic^3/\psi(\pi_1(S))$ does not have  a new simply degenerate end, but has a \lq new $\integers$-cusp' not corresponding to a $\integers$-cusp of $(\hyperbolic^3/\phi_i(\pi_1(S)))_0$.
\end{enumerate}
Since we are assuming that $\{(\mathbf m_i^-,\mathbf m_i^+)\}$ has no convergent subsequence, we can exclude the case (a)-(1).

We first assume that the condition (a)-(2)  holds.
Let $\Psi \colon S \times (0,1) \to \hyperbolic^3/\psi(\pi_1(S))$ be an orientation-preserving homeomorphism inducing $\psi$ between the fundamental groups.
For brevity of description, we now assume that a new simply degenerate end is lower.
The case when the new simply degenerate end is upper can be dealt with in the same way.
%Then there is an open incompressible subsurface $\Sigma$ in $\Sigma_-$ or $\Sigma_+$ such that $\Psi(\bar \Sigma)$ faces the new geometrically infinite end, where $\bar \Sigma$ denotes the closure of $\Sigma$, and the minimal supporting surface of its ending lamination $\mu$ coincides with $\bar \Sigma$.
%For brevity of description, we now assume that $\Sigma$ is contained in $\Sigma_-$.
%The case when $\Sigma$ lies on $\Sigma_+$ can be dealt with in the exact same way.

%Since we assumed that $\{q(m_i,n_i)\}$ converges strongly to $\psi$, there is no torus cusp in $M_\infty$.
%Therefore by \cref{z-cusp}, for every component of $\Fr \Sigma$, its length with respect to $m_i$ goes to $0$ as $i \to\infty$. 
Let $\gamma_i$ be a shortest pants decomposition of $(\Sigma_-, \mathbf m_i^-)$.
Let $\gamma_\infty$ be the Hausdorff limit of the $\gamma_i$ regarded as geodesic laminations.
%Suppose that every minimal component of $\gamma_\infty$ is a simple closed curve.
%In this case, if there is no extra spiralling leaves in $\gamma_\infty$, then $\Sigma_-$ faces a geometrically finite end in $M_\infty$, contradicting our assumption that there is a lower new geometrically infinite end.
%If $\gamma_\infty$ contains a leaf spiralling around a minimal component, then by \cref{torus}-(2), $M_\infty$ has a torus cusp, contradicting our assumption that $M_\infty$ is a strong limit.
Since we assumed that there is a new lower simply degenerate end,  by \cref{simply deg}, there is a minimal component $\ell$ of $\gamma_\infty$ which is the ending lamination of such an end.
Let $\Sigma$ be the minimal supporting surface of $\ell$.
%Then $\ell$ is the ending lamination of a lower geometrically infinite end facing $\Sigma$ by \cref{simply deg}.
Let $\lambda$ be a measured lamination supported on $\ell$.
%Since the length of every frontier component of $\Sigma$ with respect ot $m_i$ goes to $0$, the total length of $\gamma_i$ is bounded above as $i \to \infty$.
%Reference Minsky etc?
%Then there is a sequence of  positive scalars $\{r_i\}$ going to $0$ such that $\{r_i\gamma_i|
%\Sigma\}$ converges to $\lambda$.
%%Since the length of $\gamma_i$ with respect to $m_i$ is bounded as $i \rightarrow \infty$, so is $L(\phi_i, \gamma_i)$ by Sullivan's lemma (\cite{EM}) or  Ohshika \cite[Proposition 2.2]{OhI}, which is a generalisation of  Bers \cite[Theorem 3]{BeA}.
%%This implies that $\lim_{i \to \infty} L(\phi_i, r_i \gamma_i) \to 0$.
%%By the continuity of length function (see \cref{cont}) and the uniqueness of ending lamination (see \cref{EL}), we see that $\ell=|\lambda|$ is the ending lamination of .
%Then, by \cref{shortest int}, we see that $i(\nu, \lambda)\leq\lim_{i \to \infty}r_ii(b\circ q(m_i,n_i), \gamma_i)=0$.
%Since $\lambda$ is a minimal lamination whose minimal supporting surface is $\Sigma$, either the support of  a component  of $\nu$ coincides with $\mu$ supporting $\lambda$ or $\nu$ is disjoint from $\Sigma$.
Let $\delta >0$ be the positive constant given  in \cref{homotopic geodesic} for $\epsilon=1$.
Now, take an essential  simple closed curve $d$ on $\Sigma$ such that $\iota(\nu, d)<\delta$.
Since we have $\iota(\lambda , d) >0$ from the fact that $\Sigma=S(\lambda)$, and $\gamma_i$ converges to $\gamma_\infty$ containing $\ell=|\lambda|$ in the Hausdorff topology, the intersection number $\iota(\gamma_i, d)$ goes to $\infty$, which implies the length of $d$ with respect to $\mathbf m_i^-$ also goes to $\infty$, for $\gamma_i$ is a shortest pants decomposition.
On the other hand, $\{b\circ q(\mathbf m_i^-,\mathbf m_i^+)\}$ converges to $\nu$ by our assumption.
Then, by setting $\mu_\infty^1$ in the statement of \cref{homotopic geodesic} to be $\nu$ and $\mu^2_\infty$ to be empty,  it follows that the translation length of  $\phi_i(d)$ goes to $\infty$ as $i \longrightarrow \infty$, contradicting the assumption that $\{\phi_i\}$ converges.

We now assume the condition (a)-(3) that $\hyperbolic^3/\psi(\pi_1(S))$ has a new $\integers$-cusp, represented by a non-contractible, non-peripheral  simple closed curve $c$ on $\Sigma_-\sqcup \Sigma_+$, but does not have a new simply degenerate end.
In the same way as in the preceding paragraph, we can assume that $c$ lies on $\Sigma_-$.
Since we assumed that there is no new simply degenerate end, if the $\integers$-cusp represented by $\psi(c)$ touches a simply degenerate end, it corresponds to a simply degenerate end of $(\hyperbolic^3/\phi_i(\pi_1(S))_0$, which implies that $\phi_i(c)$ is parabolic, contradicting our assumption that $c$ is a new parabolic curve.
Therefore, the lower $\integers$-cusp represented by $c$ has geometrically finite ends on its both sides, which may coincide.
Therefore, the lower boundary of the convex core of $\hyperbolic^3/\psi(\pi_1(S))$ has bending angle $\pi$ along $c$.
Since we assumed that $\{\phi_i\}$ converges to $\psi$ strongly, the convex core of $\hyperbolic^3/\phi_i(\pi_1(S))$ converges to that of $\hyperbolic^3/\psi(\pi_1(S))$ (see \cite[Proof of Theorem 5]{OhC}), and hence the bending angle along $c \subset \Sigma_-$ of the convex core of $\hyperbolic^3/\phi_i(\pi_1(S))$ converges to $\pi$.
This means that $\{b\circ q(\mathbf m_i^-,\mathbf m_i^+)\}$ tends to a point in $D$, contradicting our assumption.
Thus,  in every sub-case of the case (a), we have obtained a contradiction.

We state what we have proved in the last paragraph as a lemma to refer to it in the following section.

\begin{lemma}
\label{new parabolic}
Suppose that $\{\phi_i\}$ converges to $\psi$ strongly, and that  $\hyperbolic^3/\psi(\pi_1(S))$ has a $\integers$-cusp whose core curve is a simple closed curve $c$ on a component $\Sigma$ of $\Sigma_- \sqcup \Sigma_+$.
Suppose furthermore that the $\integers$-cusp has geometrically finite ends on its both sides.
Then the bending angle on the boundary  of the convex core of $\hyperbolic^3/\phi_i(\pi_1(S))$ along $c$ converges to $\pi$.
\end{lemma}
\subsection{Case (b)}
\label{b}
%As can be seen in the classification of geometric limits given in Ohshika-Soma \cite[Theorem A]{OS} and Ohshika \cite[Theorem 4.2]{OhT}, we see that 
By \cref{strong}, if $\{q(\mathbf m_i^-,\mathbf m_i^+)\}$ converges algebraically but not strongly, then the geometric limit either has a $\integers \times \integers$-cusp or a new  simply degenerate end.
We note that the following arguments only use the assumption that $\nu$ satisfies the condition (c) of \cref{main}.
Later in \S 6, we shall use the arguments again in the case when $\nu$ satisfies (c) but not (b) of \cref{main}.

%Recall that the basepoint $x_i$ was taken to lie on the pleated surface $f_i(S)$.
%The pleated surfaces $f_i$ converge to a pleated surface $f_\infty \colon S \to M_\infty$, which can be lifted to a map to the algebraic limit, and is called an algebraic locus in \cite{OhT}.
%The limit basepoint $x_\infty$ lies on its image $f_\infty(S)$.
By \cref{OS}, the geometric limit $M_\infty$ is topologically embedded in $S \times (0,1)$.
We regard $M_\infty$ as embedded in $S \times (0,1)$ from now on, and we call the direction of $S \times \{t\}$ {\em horizontal}.
By \cref{alg loc}, there is an algebraic locus $f_\infty \colon S \to M_\infty$ which can be lifted to an immersion into $\hyperbolic^3/\psi(\pi_1(S))$ inducing $\psi$ between the fundamental groups such that $f_\infty(S)$ is horizontal except for the part where it goes around a torus cusp.
We recall that the pull-back of $f_\infty$ to $\hyperbolic^3/\phi_i(\pi_1(S))$ by an approximate isometry between $\hyperbolic^3/\phi_i(\pi_1(S))$ and $M_\infty$ induces the isomorphism $\phi_i$ between the fundamental groups.

A geometrically infinite end or a torus cusp is situated either above $f_\infty(S)$ or below $f_\infty(S)$, except for the case of a torus cusp around which $f_\infty(S)$ goes.
When $f_\infty(S)$ goes around a torus cusp $T$, we regard $T$ as being situated both above and below $f_\infty(S)$.
We say that a simply degenerate end or a torus cusp is {\em nearest to} $f_\infty(S)$ when  pleated surfaces tending to the end can be homotoped into $f_\infty(S)$ in $M_\infty$ for a simply degenerate end, and when a longitude can be homotoped into $f_\infty(S)$ for a torus end.
Since a wild geometrically infinite end is an accumulation of simply degenerate ends, it cannot be nearest to $f_\infty(S)$.
\begin{claim}
\label{nearest}
Unless $\{\phi_i\}$ converges to $\psi$ which has no new simply degenerate ends, there is either a new simply degenerate end  or a torus cusp, which is nearest to $f_\infty(S)$.
If a simply degenerate end is nearest to $f_\infty(S)$, then it can be lifted to the algebraic limit.
\end{claim}
\begin{proof}
Since $M_\infty$ is embedded in $S \times (0,1)$,  a simply degenerate end or a longitude of a torus end can be vertically homotoped into $f_\infty(S)$ unless there is another simply degenerate end or a torus end which impedes this.
Since ends or torus cusps can accumulate only into a horizontal surface containing an end, there cannot be ends which accumulate into $f_\infty(S)$.
Therefore, there must be a simply degenerate end or a torus cusp which can be homotoped into $f_\infty(S)$ without being obstructed by other ends.
This shows the first statement.

If a simply degenerate end is nearest to $f_\infty(S)$, it is homotopic into $f_\infty(S)$, and hence can be lifted to the algebraic limit.
\end{proof}

\subsubsection{Nearest new simply degenerate end.}
\label{nearest sd end}
Suppose that there is a new simply degenerate end $E$ which is nearest to $f_\infty(S)$.
We assume that $E$ is situated above $f_\infty(S)$.
The case when $E$ is situated below $f_\infty(S)$ can be dealt with in the same way just by turning everything upside down.
The end $E$ has a neighbourhood of the form  $\Sigma \times (s,t)$ for an incompressible subsurface $\Sigma$ of $S$, where $\Sigma\times  \{t\}$ corresponds to the end $E$ (Figure \ref{fNearestSimplyDegenerateEnd}).
Let $\delta$ be the constant given in \cref{homotopic geodesic} for $\epsilon=1$.
Take a simple closed curve $d$ contained in $\Sigma$ satisfying $\iota(d, \nu)<\delta$.
Let $\lambda$ be a measured lamination on $\Sigma$ whose support is the ending lamination of $E$.
Take a shortest pants decomposition $C_i$ of $(\Sigma_+, \mathbf m_i^+)$.
Then by \cref{simply deg}, we see that $C_i|\Sigma$ converges in the Hausdorff topology to a geodesic lamination whose only minimal component is $|\lambda|$.
Since $\iota(d, \lambda)>0$, we have $\iota(C_i, d) \longrightarrow \infty$ and see that the length of $d$ with respect to $\mathbf m_i^+$ goes to $\infty$.
On the other hand, $\{\iota(d, b\circ q(\mathbf m_i^-,\mathbf m_i^+))\}$ is bounded since $b\circ q(\mathbf m_i^-,\mathbf m_i^+)$ converges to $\nu$.
We set  $\mu_\infty^1$ in the statement of \cref{homotopic geodesic}  to be this $\nu$ and $\mu_\infty^2$ to be empty.
Then we see that the translation length of $\phi_i(d)$ goes to $\infty$, contradicting our assumption that $\{\phi_i\}$ converges.

%%%%

\begin{figure}
\begin{overpic}[scale=.15%, grid,tics=10
] {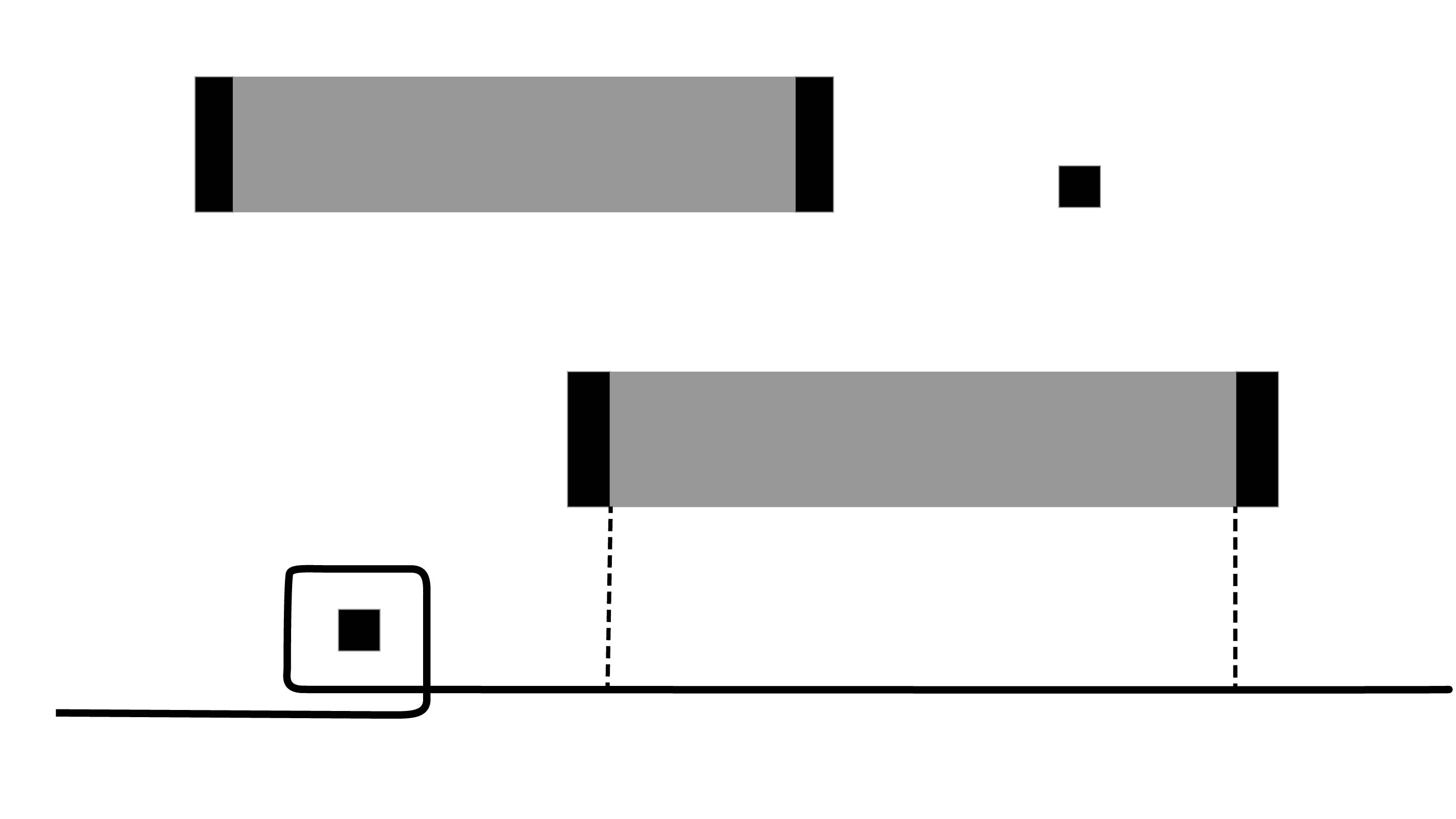} % figure file
 %   \put( , ){\textcolor{}{$$}}  
    \put(61 ,24 ){$E$}  
      \put( 0, 10){$f_\infty(S)$}  
      \end{overpic}
\caption{A nearest end above is a simply degenerate end}\label{fNearestSimplyDegenerateEnd}
\end{figure}
%%%%%

\subsubsection{Nearest torus cusp.}
\label{nearest torus}
Let $T$ be such a nearest torus cusp.
Again, the case when there is a nearest torus cusp $T$ below $f_\infty(S)$ can also be dealt with in the same way, by turning everything upside down.
Let $l$ be a simple closed curve on $\Sigma_+$ whose image under $f_\infty$ is homotopic to  a longitude of $T$.

Let $\delta>0$ be the constant given in \cref{homotopic geodesic} for $\epsilon=1$, and take a simple closed curve $d$ on $\Sigma_+$ with $\iota(d, l) >0$ and $\iota(d, \nu\setminus l) < \delta$. 
Since $l$ is homotopic to a longitude of a torus cusp in $M_\infty$,  by \cref{torus}-(1), the length of $l$ with respect to $\mathbf m_i^+$  is bounded from below by a positive constant.
Consider the shortest pants decomposition $P_i$ of $(\Sigma_+, \mathbf m_i^+)$, and extend it to a shortest marking $M_i$ of $\Sigma_+$.
Then by  \cref{torus}-(2),  if we choose a simple closed curve $k$ on $S$ intersecting $l$ essentially, we have $d_{A(l)}(M_i,k) \longrightarrow \infty$.
This means that unless $P_i$ contains $l$ for all $i$ after passing to a subsequence, there is a component $a_i$ of $P_i$ which spirals around $l$ more and more as $i \longrightarrow \infty$.
Then, since $\iota(P_i, d) \geq \iota(a_i, d)$ and the right hand side goes to $\infty$, the length of $d$ with respect to $\mathbf m_i^+$ goes to $\infty$.

In the case when $P_i$ contains $l$ for all  $i$, there is a curve $t_i$ in $M_i$ with $\iota(l,t_i)>0$ which is shortest among such curves.
Since we are assuming that the length of $l$ with respect to $\mathbf m_i^+$ is bounded away from $0$, the length of $t_i$ with respect to $\mathbf m_i^+$ is bounded above.
Then by \cref{torus}-(2), we see that $t_i$ spirals around $l$ more and more as $i \longrightarrow \infty$.
It follows that $d_{A(l)}(t_i, d)$ goes to $\infty$. 
Since $t_i$ has bounded length and the length of $l$ with respect $\mathbf m_i^+$ does not go to $0$, we see that the length of $d$ with respect to $\mathbf m_i^+$ goes to $\infty$ also in this case.

%In the same way as in the preceding paragraph, $i(d, b \circ q(m_i, n_i))$ is bounded as $i \to \infty$.
By the observation above,  the closed geodesic representing $d$ in $(\Sigma_+, \mathbf m_i^+)$ spirals around $l$ more and more  as $i \longrightarrow \infty$.
This implies that $\angle_{(\Sigma_+, \mathbf m_i^+)}(d, l)$ goes to $0$.
We set $\mu_\infty^1$ to be $\nu \setminus l$ and $\mu_\infty^2$ to be $l$ with a positive weight given by $\nu$ if $l$ is contained in $|\nu|$, and $\mu_\infty^1$ to be $\nu$ and $\mu_\infty^2$ to be empty otherwise.
Then we  apply \cref{homotopic geodesic}, and see that the translation length of $\phi_i(d)$ goes to $\infty$, contradicting our assumption that $\{\phi_i\}$ converges.

\subsection{Case (c)}
\label{c}
Suppose that $\{q(\mathbf m_i^-, \mathbf m_i^+)\}$ does not converge algebraically.
By considering efficient pleated surfaces as in Thurston \cite{Th2}, the following can be proved.
See Ohshika \cite[Theorem 3.1]{OhJ} for a complete proof.
\begin{lemma}
\label{vertical}
There is a vertical codimension-1 lamination $L$ properly embedded in $S \times [0,1]$, which is disjoint from both $\lambda_-$ and $\lambda_+$, such that for any sequence of weighted simple closed curves $s_i \gamma_i$ on $S$ converging to a measured lamination $\mu$, if $\iota(L,\mu)>0$, then $s_i \len(\phi_i(\gamma_i)) \longrightarrow \infty$, where $\len(\phi_i(\gamma_i))$ denotes the  length of the closed geodesic representing $\phi_i(\gamma_i)$ in $\hyperbolic^3/\phi_i(\pi_1(S))$.
Indeed, there is a pleated surface called an \lq efficient' pleated surface in $\hyperbolic^3/\phi_i(\pi_1(S))$ on which the growth of the length of every measured lamination is comparable to that in $\hyperbolic^3/\phi_i(\pi_1(S))$.
\end{lemma}
%In particular, $L$ cannot intersect ending laminations essentially.
%such that $m_i$ converges in the Thurston compactification of $\teich(\Sigma_- \sqcup \Sigma_+)$ to $L \cap (\Sigma_- \sqcup \Sigma_+)$.

By the condition (c) of \cref{main}, which $\nu$ must satisfy, there is a component $\nu_0$ of $\nu$ intersecting a component $L_0$ of $L$ essentially.
Let $\Sigma$ be a component of $\Sigma_- \sqcup \Sigma_+$ containing $\nu_0$, and take  a shortest pants decomposition $P_i$ of $(\Sigma,(\mathbf m_i^-, \mathbf m_i^+)|\Sigma)$.
Since $P_i$ is a shortest pants decomposition, by Sullivan's lemma or \cite[Proposition 2.1]{OhI}, $\{\len(\phi_i(P_i))\}$ is bounded.
%, which also implies, by the last statement of \cref{vertical}, that the lengths of $P_i$ on the efficient pleated surfaces are also bounded.
By \cref{vertical} we can  apply the proof of  \cite[Lemma 5.7]{OhT} using  the length in $\hyperbolic^3/\phi_i(\pi_1(S))$ instead of $(\mathbf m_i^-, \mathbf m_i^+)|\Sigma$, we see that the Hausdorff limit of $P_i$ contains $L_0 \cap \Sigma$.
This implies in turn that there is a sequence of positive numbers $r_i$ tending to $0$ and a component $P_i'$ of $P_i$ such that $r_i P'_i$ converges to a measured lamination containing $L_0 \cap \Sigma$ on $(\Sigma,(\mathbf m_i^-, \mathbf m_i^+)|\Sigma)$ unless $L_0 \cap \Sigma$ coincides with $P_i'$ for every sufficiently large $i$.
Even in the latter case, unless the length of $L_0 \cap \Sigma$ with respect to $(\mathbf m_i^-, \mathbf m_i^+)|\Sigma$ goes to $0$, by letting $P_i''$ be the shortest simple closed curve transverse to $P_i'$, we can find $s_i$ going to $0$ such that $s_i P_i'' \longrightarrow L_0 \cap \Sigma$.
Since the length of $P_i''$ with respect to $(\mathbf m_i^-, \mathbf m_i^+)|\Sigma$ is bounded, abusing the symbols, we denote $s_i P_i''$ also by $r_i P_i'$ in this case.
The properties which we shall use below are that $r_i \longrightarrow 0$ and that $\{\len_{(\mathbf m_i^-,\mathbf m_i^+)|\Sigma}(P_i)\}$ is bounded.

If the length of $L_0 \cap \Sigma$ with respect to $(\mathbf m_i^-, \mathbf m_i^+)|\Sigma$ goes to $0$, the efficient pleated surface is  pinched  along  $L_0\cap \Sigma$.
% or twisted around $L_0 \cap \Sigma$ more and more  as $i \to \infty$.
This implies that $\hyperbolic^3/\phi_i(\pi_1(S))$ is also pinched along $L_0$,
% or twisted around $L_0$ indefinitely, 
and hence $(\Sigma, (\mathbf m_i^-, \mathbf m_i^+)|\Sigma)$ is also pinched along $L_0\cap \Sigma$.
% or twisted around $L_0\cap \Sigma$ indefinitely.
Since $\nu_0$ intersects $L_0$, this contradicts \cref{shortest int}.

It remains to consider the case when $\{r_i P_i\}$ converges to $L_0 \cap \Sigma$ with $r_i \longrightarrow 0$.
Since $\iota(\nu_0, L_0) >0$ and $r_i$ goes to $0$, we see that $\iota(P_i', \nu_0) \longrightarrow \infty$.
This contradicts \cref{shortest int}, for $\{\len_{(\mathbf m_i^-,\mathbf m_i^+)|\Sigma}(P_i)\}$ is bounded.
Thus we have completed the proof of the case (c).

We note that we did not use the assumption (b) of \cref{main}.
We state what we have proved in the case (c) as a lemme for later use.

\begin{lemma}
\label{must converge}
Suppose that $\{b \circ q(\mathbf m_i^-,\mathbf m_i^+)\}$ converges to a measured lamination $\nu$ satisfying the assumption  (c) of \cref{main}.
Then $\{q(\mathbf m_i^-,\mathbf m_i^+)\}$ converges in $\AH(S)$.
\end{lemma}
%むしろL \cap Sがsccかどうか分ける．
%Suppose that $\nu_0$ lies on a component $\Sigma$ of $\Sigma_- \sqcup \Sigma_+$.
%If $\nu_0$ is a simple closed curve, since its length with respect to $(m_i, n_i)|\Sigma$ goes to $\infty$ (for $\len(\phi_i(\nu_0))$ goes to $\infty$), there is a simple closed curve intersecting $\nu_0$ essentially whose length goes to $0$.
%This  contradicts \cref{bound int lem}.
%Therefore, $\nu_0$ is not a simple closed curve.
%
%Then, by \cref{shortest int}, there are positive numbers $r_i$ such that $\{r_i P_i\}$ converges to a measured lamination containing $\nu_0$.
%%On the other hand, for a shortest pants decomposition $C_i$ of $(\Sigma_-\sqcup \Sigma_+, (m_i, n_i))$, there are positive numbers $r_i$ such that $r_i C_i$ converges to $L \cap (\Sigma_- \sqcup \Sigma_+)$.
%Since $i(\nu_0, L)>0$, we see that the length of realisation of $\nu_0$ in $\hyperbolic^3/\phi_i(\pi_1(S))$ goes to $\infty$.
%
%This implies that  $\nu_0$ cannot contain a component of $P_i$ for large $i$.
%Therefore, $r_i$ must go to $0$.
%Furthermore, again since  $i(\nu_0, L)>0$, we see that $r_i\len(\phi_i(P_i))$ goes to $\infty$.
%On the other hand, since $\{len(\phi_i(P_i))\}$ is bounded and $r_i \to 0$, we have $r_i\len(\phi_i(P_i)) \longrightarrow 0$.
%This is a contradiction.
%%By \cref{shortest int}, $i(
%%This implies that $i(C_i, \nu)$ diverges at the order of $r_i^{-1}$, contradicting \cref{bound int lem}.

\section{Homotopy to a degree-1 map}
\label{degree 1}
In this section, we shall prove that $b \circ q$ is a degree-1 map to $\ml(\Sigma_-) \times \ml(\Sigma_+) \setminus D$ by constructing a homotopy in the one-point compactification of $\ml(\Sigma_-)\times \ml(\Sigma_+) \setminus D$ between the map induced from $b \circ q$ and a degree-1 map.
% from $\teich(\Sigma_1) \times \teich(\Sigma_2)$ to $\ml(S) \times \ml(S)$.
First, we shall define compactification where a homotopy will take place.

\begin{definition}
\label{compactification}
We let $\ct$ be the one-point compactification of $\teich(\Sigma_-) \times \teich(\Sigma_+)$, and $\cml$ the one-point compactification of $\ml(\Sigma_-) \times \ml(\Sigma_+) \setminus D$.
Since $b\circ q$ is proper, it induces a continuous map $\cb \colon \ct \to \cml$.
\end{definition}
The following is immediate from the definition of one-point compactification.
\begin{lemma}
The map $b \circ q$ has degree 1 if and only if $\cb$ has degree 1.
\end{lemma}

Therefore, we have only to show that $\cb$ has degree $1$.
For that, we shall construct a homotopy  in an open set from $\cb$ to a locally degree-1 map.
To construct the latter map, we shall make use of the following homeomorphism derived from the earthquake introduced by Thurston (see Thurston \cite{ThE} and Kerckhoff \cite{Ke}).

\begin{definition}
\label{standard homeo}
Fix $g_- \in \teich(\Sigma_-)$ and $g_+ \in \teich(\Sigma_+)$.
For $j=-,+$, we let $E_j \colon \ml(\Sigma_j) \to \teich(\Sigma_j)$ be the left earthquake map, that is, a homeomorphism sending a measured lamination $\lambda \in \ml(\Sigma_j)$ to the marked hyperbolic structure obtained  by the left earthquake along $\lambda$ on $(\Sigma_j, g_j)$.
Then, we have a homeomorphism $E_- \times E_+ \colon \ml(\Sigma_-) \times \ml(\Sigma_+) \to \teich(\Sigma_-) \times \teich(\Sigma_+)$.
We define $\ce \colon \teich(\Sigma_-) \times \teich(\Sigma_+) \to \ml(\Sigma_-) \times \ml(\Sigma_+)$ to be the inverse of $E_- \times E_+$.
\end{definition}

To construct a homotopy, we shall first define its support, which will be done by using an open neighbourhood of a point contained in the \lq corner' of the product of the Thurston compactifications of  the components of $\teich(\Sigma_-) \times \teich(\Sigma_+)$.
%Here we consider the Thurston compactification of  $\teich(\Sigma_-)\times \teich(\Sigma_+)$ by compactifying component by component.
We now describe it more concretely.
Let $\Sigma_1, \dots, \Sigma_n$ be the components of $\Sigma_- \sqcup \Sigma_+$ that are not thrice-punctured spheres.
We compactify each $\teich(\Sigma_j)$ by attaching $\pl(\Sigma_j)$ as its boundary.
We call their product $\prod_{j=1}^n (\teich(\Sigma_j) \cup \pl(\Sigma_j))$ the {\em Thurston compactification product} of $\teich(\Sigma_-) \times \teich(\Sigma_+)$ and denote its boundary by $\pl$.
A point of $\pl$ has a form $(x_j)_{j =1}^n$, where $x_j$ is either $\teich(\Sigma_j)$ or $\pl(\Sigma_j)$ and at least one of the $x_j$ is contained in the boundary $\pl(\Sigma_j)$.
We put the product topology on the compactification.
%identifying $\teich(\Sigma_-) \times \teich(\Sigma_+)$ with $\teich(\Sigma_- \sqcup \Sigma_+)$.
%Then we can regard  $(\prod_{j=1}^n\pl(\Sigma_j))$ as a set lying at infinity of $\teich(\Sigma_-) \times \teich(\Sigma_+)$, by defining every product of neighbourhoods of $[\Lambda_j]\ (j=1,\dots , n)$ to be a neighbourhood of $([\Lambda_j])_j$.
%We denote the latter space by $\partial \mathcal T$ for brevity.
We call the subset $\prod_{j=n}^n \pl(\Sigma_j)$ of the boundary the {\em corner} and denote it by $\plc$.
%Attaching $\pl$ does not compctify $\teich(\Sigma_-)\times \teich(\Sigma_+)$, but
A sequence $\{\mathbf m_i\}$ of $\teich(\Sigma_-) \times \teich(\Sigma_+)$ converges to a point in $\plc$ after passing to a subsequence if $\{\mathbf m_i|\Sigma_j\}$ diverges for every component $\Sigma_j$ of $\Sigma_- \sqcup \Sigma_+$ that is not a thrice-punctured sphere.

\begin{definition}
\label{Lambda}
Let $\Lambda$ be a point in $\ml(\Sigma_-\sqcup\Sigma_+)$  not contained in $D$ such that, for each component $\Sigma_j$ of $\Sigma_- \sqcup \Sigma_+$ that is not thrice-punctured sphere, the restriction $\Lambda \cap \Sigma_j$ is an arational uniquely ergodic measured lamination.
By setting its $j$-th coordinate to be $[\Lambda|\Sigma_j]$, we define $[\Lambda] \in \plc$.
%We denote by $[\Lambda] \in \pl(\Sigma_-\sqcup \Sigma_+)$ the projective class of $\Lambda$.
\end{definition}
We note that by the arationality, the condition that $\Lambda$ is not contained in $D$ is equivalent to, when $\lambda_-=\lambda_+=\emptyset$,  the condition that no two components of $\Lambda$ are homotopic, and otherwise, the condition that no component of the support of $\Lambda$ is homotopic  to an ending lamination.

\begin{definition}
\label{train track}
Let $\tau$ be a bi-recurrent train track on $\Sigma_- \sqcup \Sigma_+$ carrying $\Lambda$ by a weight system $\omega$ in such a way that $\omega$ takes a positive value on every branch of $\tau$.
We call an arc connecting two measured laminations $\mu_1, \mu_2$ carried by $\tau$ a {\em segment} when it is a linear path with regard to the weight system.
We note that this notion is independent of the choice of $\tau$ since the transition function between two weight systems is linear.

For two measured laminations $\lambda_1, \lambda_2$ carried by $\tau$ with weight systems $\omega_1, \omega_2$ respectively, we define $d_\tau(\lambda_1, \lambda_2)$ to be the sum of the differences of the weights of $\omega_1$ and $\omega_2$  on the branches of $\tau$.
\end{definition}
From now on until the end of this section, we fix a train track $\tau$ as above.

In the same way as we did for $\teich(\Sigma_-) \times \teich(\Sigma_+)$, for each component $\Sigma_j\ (j=1,\dots , n)$ of $\Sigma_- \sqcup \Sigma_+$, we consider the ray compactification of $\ml(\Sigma_j)$ and regard $\pl(\Sigma_j)$ as its boundary at infinity, and define the {\em ray compactification} of $\ml(\Sigma_-) \times \ml(\Sigma_+)$ with  boundary at infinity $\pl$ to be the product of the ray compactification of the components $\ml(\Sigma_1), \dots, \ml(\Sigma_n)$.
%We also attach $\pl$ as a corner at infinity to $\ml(\Sigma_-) \times \ml(\Sigma_+)$ in the same way as the case of Teichm\"{u}ller space, and put a topology on $\ml(\Sigma_-) \times \ml(\Sigma_+) \cup \pl$ defining a neighbourhood of a point at infinty as a product of neighbourhoods at infinity of the ray compactificaton $\ml(\Sigma_j) \sqcup\pl(\Sigma_j)$..
As before, we call $\plc$ the corner also in this ray compactification, and we see that a sequence $\{\lambda_i\}$ in  $\ml(\Sigma_-) \times \ml(\Sigma_+)$ converges to a point in $\pl$ after passing to a subsequence if $\{\lambda_i|\Sigma_j\}$ diverges for every $j=1, \dots, n$.

\begin{definition}
We call a subset $U$ of $\ml(\Sigma_-) \times \ml(\Sigma_+)$ a {\em  truncated cone} if  $U$ consists of all measured laminations $\lambda$ that satisfy the following conditions.
\begin{enumerate}[(a)]
\item
The train track $\tau$ carries $\lambda$.
\item
Each weight of $w(\lambda)$ is greater than a fixed positive constant $K$.
(Recall from \cref{train track} that we denote the weight system on $\tau$ corresponding to $\lambda$ by $w(\lambda)$.)
%\item
%The projective class $[\lambda]$ is contained in $V$.
\end{enumerate}
We say that a truncated cone $U$  is a {\em truncated cone neighbourhood}
of $[\Lambda]$ (supported on $\tau$) when there is a neighbourhood $V$ of $[\Lambda]$ in the boundary at infinity $\pl$ such that the closure of $U$ in the ray compactification $(\ml(\Sigma_-) \times \ml(\Sigma_+))$ coincides with $V$.
%, and $U$ consists of all measured laminations $\lambda$, which satisfies the following conditions.
%\begin{enumerate}[(a)]
%\item
%The train track $\tau$ carries $\lambda$.
%Recall from \cref{train track} that we denote the weight system on $\tau$ corresponding to $\lambda$ by $w(\lambda)$.
%\item
%Each weight of $w(\lambda)$ is greater than $K$.
%%\item
%%The projective class $[\lambda]$ is contained in $V$.
%\end{enumerate}

The positive weight systems on $\tau$ form an open set in $\ml(\Sigma_-) \times \ml(\Sigma_+)$ since $\tau$ is bi-recurrent, and its  ray compactification contains $[\Lambda]$ since each $\Lambda|\Sigma_j$ is arational and uniquely ergodic.
Therefore we see that (the ray compactifications of) the truncated cone neighbourhoods form a basis of neighbourhoods of $[\Lambda]$ in the ray compactification $(\ml(\Sigma_-)\times \ml(\Sigma_+)) \cup \pl$.\end{definition} 

The following is a well-known property of the earthquake map. 
(See Papadopoulos \cite{Papad} for instance.)
\begin{lemma}
\label{E at infinity}
The map $\mathcal E \colon \teich(\Sigma_-) \times \teich(\Sigma_+) \to \ml(\Sigma_-) \times \ml(\Sigma_+)$ is a homeomorphism extending continuously to the identity on the boundary $\pl$.
\end{lemma}

We now show some lemmas and their corollaries which will be used in the main step of the proof of \cref{proper degree 1}.

\begin{lemma}
\label{infinity correspondence}
Let $[\Lambda]$ is a projective lamination in $\plc$ as in \cref{Lambda}.
Let $\{\mathbf m_i\}$ be a sequence in $\teich(\Sigma_-) \times \teich(\Sigma_+)$ converging to $[\Lambda]$ in the Thurston compactification product.
Then $\{b \circ q(\mathbf m_i)\}$ also converges to $[\Lambda]$ in the ray compactification $(\ml(\Sigma_-) \times \ml(\Sigma_+))\cup \pl$.
\end{lemma}
\begin{proof}
Let $\Sigma_j$ be a component of $\Sigma_- \sqcup \Sigma_+$.
Since $\{\mathbf m_i\}$ converges to $[\Lambda]$ contained in the corner, the restriction $\{\mathbf m_i|\Sigma_j\}$ converges in the Thurston compactification to $[\Lambda|\Sigma_j] \in \pl(\Sigma_j)$.
Let $k_i$ be a shortest simple closed geodesic in $(\Sigma_j, \mathbf m_i|\Sigma_j)$.
Since $\Lambda|\Sigma_j$ is arational and uniquely ergodic, there is a sequence of positive numbers $\{s_i\}$ going to $0$ such that $\{s_i k_i\}$ converges to the measured lamination $\Lambda|\Sigma_j$.
By \cref{shortest int}, $\{\iota(b\circ q(\mathbf m_i), k_i)\}$ is bounded.
By the continuity of intersection numbers combined with the arationality and the unique ergodicity of $\Lambda|\Sigma_j$, we see that either $\{b\circ q(\mathbf m_i)\}$ converges to $s (\Lambda|\Sigma_j)$ for some positive scalar $s$, or $\{b\circ q(\mathbf m_i)\}$ converges in the ray compactification to the point at infinity $[\Lambda|\Sigma_j] \in \pl(\Sigma_j)$.

%By the properness of $b\circ q$, passing to a subsequence, we can assume that   $\{b \circ q(\mathbf m_i)\}$ either converges to a measured lamination $\mu \in D$ or diverges in $\ml(\Sigma_-) \times \ml(\Sigma_+)$.
%In the latter case, since $\{\mathbf m_i\}$ converges to a point in the corner, for every component $\Sigma_j$, the restriction $\{\mathbf m_i|\Sigma_j\}$ converges in the Thurston compactification to $[\Lambda|\Sigma_j] \in \pl(\Sigma_j)$.
%Let $k_i$ be a shortest simple closed geodesic in $(\Sigma_j, \mathbf m_i)$.
%Since $\Lambda|\Sigma_j$ is arational and uniquely ergodic, there is a sequence of positive numbers $\{s_i\}$ going to $0$ such that $\{s_i k_i\}$ converges to the measured lamination $\Lambda|\Sigma_j$.
%By \cref{shortest int}, $\{i(b\circ q(\mathbf m_i), k_i)\}$ is bounded.
%By the continuity of intersection numbers, it follows that either $\{b\circ q(\mathbf m_i)\}$ converges to $s (\Lambda|\Sigma_j)$ for some positive scalar $s$, or $\{b\circ q(\mathbf m_i)\}$ converges in the ray compactification to the point at infinity $[\Lambda|\Sigma_j] \in \pl(\Sigma_j)$.

It remains to show that the former case cannot happen.
Suppose, seeking a contradiction, that $\{b \circ q(\mathbf m_i)\}$ converges to $s (\Lambda|\Sigma_j)$.
Then, by \cref{must converge}, we see that $\{q(\mathbf m_i)|\pi_1(\Sigma_j)\}$ converges algebraically.
Let the constant $\delta>0$ given in \cref{homotopic geodesic} for  $\epsilon=1$.
We then take a simple closed curve $c$ on $\Sigma_j$ approximating $|\Lambda|\Sigma_j|$ such that $\iota(c, s\Lambda)< \delta$.
Since $\{\mathbf m_i|\Sigma_j\}$ converges to $[\Lambda|\Sigma_j]$ in the Thurston compactification, we see that $\len_{\mathbf m_i}(c)$ goes to $\infty$ as $i \longrightarrow \infty$.
By \cref{homotopic geodesic}, this implies that $\len_{q(\mathbf m_i)}(c)$ also goes to $\infty$, contradicting the fact that $\{q(\mathbf m_i)|\pi_1(\Sigma_j)\}$ converges.

Thus we have shown that $\{b\circ q(\mathbf m_i)|\Sigma_j\}$ converges to $[\Lambda|\Sigma_j]$ for every $j =1, \dots, n$, and hence $\{b\circ q(\mathbf m_i)\}$ converges to $[\Lambda]$ in the ray compactification.
%
%and hence either $\{b\circ q(\mathbf m_
%
%Since the limit $\Lambda|\Sigma_j$ of $\{s_i k_i\}$ is arational, $\{b \circ q(\mathbf m_i)|\Sigma_j\}$ must diverge in $\ml(\Sigma_j)$.
%Since this holds for every component $\Sigma_j$ of $\Sigma_- \sqcup \Sigma_+$, we see that $\{b \circ q(\mathbf m_i)\}$ converges in the ray compactification to a projective lamination $[\nu] \in \pl$.
%
%For each $i$, let $c_i$ be a component of the shortest pants decomposition of $(\Sigma_- \sqcup \Sigma_+, \mathbf m_i)$.
%Since each component of $\Lambda$ is arational and uniquely ergodic, $\{[c_i]\}$ converges to $[\Lambda]$ in $\pl$, and for each $j=1, \dots , n$ there is a sequence of positive numbers $\{r_i^j\}_i$ tending to $0$ as $i \to \infty$ such that $\{(r_i^j c_i|\Sigma_j)_{j=1}^n\}$ converges to $\Lambda$ as $i \to \infty$ in $\prod_{j=1}^n \ml(\Sigma_j)$.
%By \cref{shortest int}, $i(c_i, b\circ q(\mathbf m_i))$ is bounded as $i \to \infty$.
%By the continuity of intersection numbers, we see that $i(\Lambda, \mu)=0$ when $\{b \circ q(\mathbf m_i)\}$ converges to $\mu \in D$, and also that $i(\Lambda, \nu)=0$ when $\{b \circ q(\mathbf m_i)\}$ converges to $[\nu]$ in the ray compactification.
%In the former case, since $\Lambda$ is arational and $\Lambda \not\in D$, we get a contradiction.
%In the latter case, the arationality implies  $|\Lambda|=|\nu|$, and since $\Lambda|\Sigma_j$ was assumed to be uniquely ergodic for every $j$, we have $[\Lambda]=[\nu]$ in $\pl$.
%This completes the proof.
\end{proof}

The lemma implies the following corollary.
\begin{corollary}
\label{cont at inf}
Let $[\Lambda]$ be a projective measured lamination in $\plc$ as given in \cref{Lambda}.
For any truncated cone neighbourhood $U$ of $[\Lambda]$, there is a neighbourhood $V$ of $[\Lambda]$ in the Thurston compactification product of $\teich(\Sigma_-) \times \teich(\Sigma_+)$ such that $b \circ q(V \cap (\teich(\Sigma_-) \times \teich(\Sigma_+)))$ is contained in $U$.
\end{corollary}
\begin{proof}
Consider a sequence $\{\mathbf m_i \}$ in $\teich(\Sigma_-) \times \teich(\Sigma_+)$ converging to $[\Lambda]$ in the Thurston compactification product.
%By the properness of $b\circ q$, passing to a subsequence, we can assume that   $\{b \circ q(\mathbf m_i)\}$ either converges to a measured lamination $\mu \in D$ or diverges in $\teich(\Sigma_1) \times \teich(\Sigma_2)$ whereas converges in the Thurston compactification to a projective lamination $[\nu] \in \pl(\Sigma_1 \sqcup \Sigma_2)$.
%Let $P_i$ be a shortest pants decomposition of $(\Sigma_1 \sqcup \Sigma_2, \mathbf m_i)$.
%Then $\{[P_i]\}$ converges to $[\Lambda]$ in $\pl(\Sigma_1 \sqcup \Sigma_2)$, and there is a sequence of positive numbers $r_i$ such that $\{r_i P_i\}$ converges to $\Lambda$.
%%Reference FLP?
%By \cref{bound int lem}, $i(P_i, b\circ q(\mathbf m_i))$ is bounded as $i \to \infty$.
%By the continuity of intersection numbers, we see that $i(\Lambda, \mu)=0$ when $b \circ q(\mathbf m_i)$ converges to $\mu \in D$, and $i(\Lambda, \nu)=0$ when $b \circ q(\mathbf m_i)$ converges to $[\nu]$ in the Thurston compactification.
%In the former case, since $\Lambda$ is arational and $\Lambda \not\in D$, we get a contradiction.
%In the latter case, the arationality implies  $|\Lambda|=|\nu|$, and since $\Lambda$ was assumed to be uniquely ergodic, we have $[\Lambda]=[\nu]$.
Then by \cref{infinity correspondence}, for any given truncated cone neighbourhood $U$, the sequence $\{b \circ q(\mathbf m_i)\}$ is contained in $U$ for sufficiently large $i$.
This implies the existence of a neighbourhood of $[\Lambda]$ as desired.
\end{proof}

In a more general case where we do not assume that the limit of $\{\mathbf m_i\}$ is  $[\Lambda]$, we have the following.
\begin{lemma}
\label{gen lim}
Let $\{\mathbf m_i\}$ be a sequence in $\teich(\Sigma_-) \times \teich(\Sigma_+)$ which converges to a point $(x_j) \in \pl$ in the Thurston compactification product.
Suppose moreover that 
$\{b \circ q(\mathbf m_i)\}$ converges to a point $(y_j)$ in the ray compactification.
%, where $\mu_j$ is either a point in $\ml(\Sigma_j)$ or a point at infinity in $\pl(\Sigma_j)$ represented by $\nu_j$.
%he ray represented by a projective lamination $[\nu]$ in the ray compactification; i.e. 
%either $\{b \circ q(\mathbf m_i)\}$ converges a scalar multiple of $\nu$ or converges to a boundary point represented by $[\nu]$.
For each $x_j$ that lies on the boundary at infinity $\pl(\Sigma_j)$, let $\mu_j$ be a measured lamination with $[\mu_j]=x_j$.
Then $y_j$ lies either in $\ml(\Sigma_j)$ and $\iota(y_j, \mu_j)=0$ or on the boundary at infinity $\pl(\Sigma_j)$ and is represented by a measured lamination $\nu_j$ with $\iota(\mu_j, \nu_j)=0$.
Moreover, for every $j$ such that $y_j$ lies in $\pl(\Sigma_j)$, the point $x_j$ also lies in $\pl(\Sigma_j)$.
%by letting $\mu_j$ be a measured lamination representing $x_j$, we have $i(\mu_j, \nu_j)=0$ for every $j=1, \dots , n$.
%Moreover if $\mu_j$ lies in $\pl(\Sigma_j)$, then $\nu_j$ also lies in 
\end{lemma}
\begin{proof}
Suppose that $x_j=[\mu_j]$ lies in $\pl(\Sigma_j)$.
Let $\mu^0_j$ be a connected component of $\mu_j$.
%We note that since $\{b\circ q(\mathbf m_i)\}$ converges in the ray compactification to either a point on the ray in the direction of $\nu$ or the boundary point represented by $[\nu]$, there are bounded positive numbers $s_i$  such that $\{s_i b \circ  q(\mathbf m_i)\}$ converges to $\nu$ in $\ml(\Sigma_-) \times \ml(\Sigma_+)$.
%Let $\mu_0$ be a component of $\mu$.
If $\mu_j^0$ is a simple closed curve, either (a) $\len_{\mathbf m_i}(\mu_j^0) \rightarrow 0$ or (b) there are simple closed curves $d_i$ on $\Sigma_j$ with bounded $\len_{\mathbf m_i}(d_i)$ and positive numbers $r_i \longrightarrow 0$ such that $\{r_i d_i\}$ converges to a measured lamination $\hat \mu_j$ containing $\mu_j^0$.
The latter curve $d_i$ can be chosen to be a shortest simple closed curve on $\Sigma_j$ with respect to $\mathbf m_i$ intersecting $\mu_j^0$.
Moreover, if $\mu_j^0$ is not a simple closed curve, %by choosing $d_i$ to be a shortest simple closed curve intersecting $\mu_j$ transversely, 
the condition (b) always holds.

In the case (a), \cref{bound int lem} implies that $\iota(\mu_j^0, b \circ q(\mathbf m_i)|\Sigma_j) \longrightarrow 0$, which implies that $\iota(\mu_j^0, y_j)=0$  when $y_j$ lies in $\ml(\Sigma_j)$ and $\iota(\mu_j^0, \nu_j)=0$ when $x_j=[\nu_j]$ lies on the boundary at infinity, by the continuity of the intersection number.
In the case (b), \cref{bound int lem} implies that $\{\iota(d_i, b\circ q(\mathbf m_i))\}$ is bounded.
Since $r_i$ tends to $0$, by the continuity of the intersection number, we have  $\iota(\hat\mu_j, y_j)=0$, and hence $\iota(\mu^0_j, y_j)=0$ when $y_j$ lies in $\ml(\Sigma_j)$, and in the same way, $\iota(\mu^0_j, \nu_j)=0$ when $x_j=[\nu_j]$ lies on the boundary at infinity.
Thus,  we have $\iota(\mu_j^0, y_j)=0$ or $\iota(\mu^0_j, \nu_j)=0$ for every connected component $\mu_j^0$ of $\mu_j$, which implies that $\iota(\mu_j,y_j)=0$ or $\iota(\mu_j, \nu_j)=0$.

To show the last statement, suppose that $x_j$ lies in $\teich(\Sigma_j)$.
Then, by \cref{vertical}, we see that $\{q(\mathbf m_i)|\pi_1(\Sigma_j)\}$ converges.
Since $\{\mathbf m_i|\Sigma_j\}$ converges by assumption,  the boundary component $\Sigma_j^i$ of the convex core $C(\hyperbolic^3/q(\mathbf m_i))$ corresponding to $\Sigma_j$ converges geometrically to a boundary component $\Sigma_j^\infty$ of the convex core of the geometric limit, which is homotopic to an algebraic locus of $\Sigma_j$ (see the argument of \cite[p.103]{OhQ}).
Then the $j$-th component of $b\circ q(\mathbf m_j)$ converges to the bending lamination of $\Sigma_j^\infty$, and hence $y_j$ must be inside $\ml(\Sigma_j)$.
This shows that if $y_j$ lies on the boundary at infinity, then so does $x_j$.
\end{proof}
By a similar argument, we can also show the following proposition.
\begin{proposition}
\label{partial converse}
Let $[\Lambda]$ be a projective lamination as in \cref{Lambda}.
Let $\{\mathbf m_i\}$ be a sequence in $\teich(\Sigma_-) \times \teich(\Sigma_+)$ such that $\{b \circ q(\mathbf m_i)\}$ converges to a point at infinity represented by $[\Lambda]$ in the ray compactification.
Then $\{\mathbf m_i\}$ converges to $[\Lambda]$ in the Thurston compactification product of $\teich(\Sigma_-) \times \teich(\Sigma_+)$.
\end{proposition}
\begin{proof}
We have only to show that any subsequence of $\{\mathbf m_i\}$ has a subsequence converging to $[\Lambda]$ in the Thurston compactification product.
Passing to a subsequence, we can assume that $\{\mathbf m_i\}$ converges to either a point $\mathbf n$ in $\teich(\Sigma_-) \times \teich(\Sigma_+)$ or a projective lamination $(y_j)\in \pl$ in the Thurston compactification product.
In the former case, by the continuity of the function $b$ due to \cite{KS}, $\{b\circ q(\mathbf m_i)\}$ converges to $b\circ q(\mathbf n)$, contradicting our assumption.

Suppose that $\{\mathbf m_i\}$ converges to $(y_j)$ in the Thurston compactification product.
Since $[\Lambda]$ lies in the corner, and $\Lambda \cap \Sigma_j$ is arational and uniquely ergodic for every $j$, by \cref{gen lim}, we have %$i(\mu, \Lambda)=0$, and by the arationality and the unique ergodicity of $\Lambda$, 
we have $y_j=\Lambda \cap \Sigma_j$, and hence $(y_j)=[\Lambda]$.
This completes the proof.
\end{proof}

Since truncated cone neighbourhoods form a basis of neighbourhoods of $[\Lambda]$ as remarked before, we have the following corollary.
\begin{corollary}
\label{cont inv at inf}
Let $V$ be a neighbourhood of $[\Lambda]$ in the Thurston compactification product of $\teich(\Sigma_-) \times \teich(\Sigma_+)$ for $[\Lambda]$ as in \cref{Lambda}.
Then, there is a truncated cone neighbourhood $U$ of $[\Lambda]$ in the ray compactification such that $(b\circ q)^{-1}(U)$ is contained in $V \cap (\teich(\Sigma_-) \times \teich(\Sigma_+))$.
\end{corollary}
%\begin{proof}
%Suppose that there does not exist such a truncated cone neighbourhood $U$.
%Since the truncated cone neighbourhoods form a basis of the neighbourhoods around the point at infinity $[\Lambda]$, there exists a sequence of point $\{\mathbf m_i\}$ in $\teich(\Sigma_1) \times \teich(\Sigma_2)$ lying outside $V$ such that $\{b\circ q(\mathbf m_i)\}$ converges to $[\Lambda]$ in the ray compactification of $\ml(\Sigma_1 \sqcup \Sigma_2)$.
%Then \cref{partial converse}  implies that $\{\mathbf m_i\}$ converges to $[\Lambda]$ in the Thurston compacitification.
%Since $V$ is a neighbourhood of $[\Lambda]$, this is a contradiction.
%Passing to a subsequence, we can assume that $\{\mathbf m_i\}$ converges to either a point $\mathbf n$ in $\teich(\Sigma_1) \times \teich(\Sigma_2)$ or a projective lamination $[\mu]\in \pl(\Sigma_1 \sqcup \Sigma_2)$ in the Thurston compactification.
%
%In the former case, by the continuity of $\{b\circ q(\mathbf m_i)\}$ converges to $b\circ q(\mathbf n)$, contradicting our assumption.
%\end{proof}

By combining these results, we obtain the following technical proposition, which constitutes an essential step for our construction of a homotopy.
\begin{proposition}
\label{three rings}
Let $\Lambda$ be a measured lamination given in \cref{Lambda}.
Then there are three nested truncated cone neighbourhoods $U_0 \subset U_1 \subset U_2$ which satisfy the following.
\begin{enumerate}
\item Every measured lamination in  $U_2$ satisfies the condition (c) of \cref{main} (with $\lambda_-$ and $\lambda_+$).
\item Let $V_1$ be $(b\circ q)^{-1}(U_1)$.
Then there is an open set $V_2$ containing the closure $\bar V_1$ such that both $b \circ q(V_2)$ and $\mathcal{E}(V_2)$ are contained in $U_2$.
\item Neither $\mathcal E$ nor $b \circ q$ maps a point outside $V_1$into $U_0$.
\item For any point $m \in V_2 \setminus V_1$, the segment connecting $b \circ q(m)$ and $\mathcal E(m)$ is disjoint from $U_0$.
\end{enumerate}
\end{proposition}
\begin{proof}
%We shall first show that we can choose open-ring neighbourhoods $U_1, U_2$ of $[\Lambda]$ such that $\mathcal{E}(\overline{(b\circ q)^{-1}(U_1)})$ is contained in $U_2$.
%If a measured lamination $\mu$ on $\Sigma_- \sqcup \Sigma_+$ is contained in $D$ but does not have a compact leaf with weight larger than or equal to $\pi$, 
If a measured lamination on $\Sigma_-\sqcup \Sigma_+$ does not satisfy the condition (c) of \cref{main}, then either it has a component homotopic to a component of $\lambda_-$ or $\lambda_+$, or it has two components which are homotopic in $S \times [0,1]$.
Since $\Lambda$ is arational and is not contained in $D$, every sufficiently small truncated cone neighbourhood of $[\Lambda]$ contains no such measured laminations.
Therefore, by choosing a truncated cone neighbourhood $U_2$ to be sufficiently small, the condition (1) is satisfied.

%Take any sufficiently small truncated con neighbourhood $U_2$ of $[\Lambda]$ associated with a train track $\tau$ carrying $\Lambda$ so that it satisfies (1) as above.
By \cref{E at infinity,cont at inf}, there is a neighbourhood $V_2$ of $[\Lambda]$ in the Thurston compactification product such that both $\mathcal E(V_2)$ and $b \circ q(V_2)$ are contained in $U_2$.
By \cref{E at infinity,cont inv at inf}, we can take a neighbourhood $U_1'$ of $[\Lambda]$ such that both $\mathcal E^{-1}(U_1')$ and $(b\circ q)^{-1}(U_1')$ are contained in $V_2$.
Again by \cref{E at infinity,cont at inf,cont inv at inf}, we can take a neighbourhood $V_1'$ of $[\Lambda]$ in the Thurston compactification product such that $\overline{V_1'}$ is contained in $\mathcal E^{-1}(U_1')\cap (b\circ q)^{-1}(U_1')$, and a truncated cone neighbourhood $U_1$ of $[\Lambda]$ in the ray compactification such that both $\mathcal E^{-1}(U_1)$ and $(b\circ q)^{-1}(U_1)$ are contained in $V_1'$.
These $U_1, U_2, V_2$ and $V_1=(b\circ q)^{-1}(U_1)$ satisfy the condition (2).

Now, we shall show that we can take a truncated cone neighbourhood $U_0$ of $[\Lambda]$ in the ray compactification satisfying (3) and (4).
Let $\{\mathbf m_i\}$ be an arbitrary sequence in $\teich(\Sigma_-) \times \teich(\Sigma_+)$ lying outside $V_1$.
Since $V_1$ is a neighbourhood of $[\Lambda]$ in the Thurston compactification product, \cref{E at infinity} implies that $\{\mathcal E(\mathbf m_i)\}$ cannot converge to $[\Lambda]$ in the ray compactification.
Similarly, by \cref{partial converse}, we see that $\{b\circ q(\mathbf m_i)\}$ cannot converge to $[\Lambda]$ in the ray compactification either.
Therefore, if we take $U_0$ to be a sufficiently small truncated cone neighbourhood of $[\Lambda]$, then the condition (3) is satisfied.

We next show that $U_0$ can be taken to satisfy the condition (4).
Let $\{\mathbf m_i\}$ be a sequence in $V_2 \setminus V_1$, and $\lambda_i$  a point on the segment connecting $b\circ q(\mathbf m_i)$ and $\mathcal E(\mathbf m_i)$ such that $\{\lambda_i\}$ converges to $[\Lambda]$ in the ray compactification.
Taking  a subsequence, we can assume that either $\{\mathbf m_i\}$ converges to a point $\mathbf n \in \teich(\Sigma_-) \times \teich(\Sigma_+)$ or does not have a convergent sequence inside $\teich(\Sigma_-) \times \teich(\Sigma_+)$, and converges  to  a point $(w_j) \in \pl$ in the Thurston boundary.
The boundary point $(w_j) \in \pl$ in the latter case is distinct from $[\Lambda]$ since $\{\mathbf m_i\}$ lies outside $V_1$.
In the former case, passing to a subsequence $\{\lambda_i\}$ converges to a point inside $\ml(\Sigma_-) \times \ml(\Sigma_+)$ which lies on a segment connecting $\mathcal E(\mathbf n)$ and $b \circ q(\mathbf n)$, contradicting our assumption.
In the latter case, by \cref{E at infinity}, we see that $\{\mathcal{E}(\mathbf m_i)\}$ converges to $(w'_j) \in \pl$ such that $w_j'=\mathcal (w_j)$ if $w_j \in \teich(\Sigma_j)$ and $w_j'=w_j$ otherwise.
The sequence $\{b \circ q(\mathbf m_i)\}$ converges to a point $(y_j)$ in the ray compactification passing to a subsequence.
If there is $j$ such that  $w_j'$  lies in $\ml(\Sigma_j)$, then $y_j$ must  lie in $\ml(\Sigma_j)$.
Then the limit of $\lambda_j$ also lies in $\ml(\Sigma_j)$ and contradicts the assumption that $[\Lambda]$ is a corner point.
Therefore, we have $(w_j')=(w_j)$ and $w_j \in \pl(\Sigma_j)$.
Let $\mu$ be a measured lamination on $\Sigma_- \sqcup \Sigma_+$ representing $(w_j)$.
By \cref{gen lim} we have $\iota(y_j, \mu|\Sigma_j)=0$ if $y_j$ lies inside $\ml(\Sigma_j)$ and is represented by $\mu_j' \in \ml(\Sigma_j)$ with $y_j=[\mu_j']$ and $\iota(\mu|\Sigma_j, \mu_j')=0$ if $y_j$ lies in $\pl(\Sigma_j)$.
Recall that $\lambda_i$ lies on the segment between $\mathcal E(\mathbf m_i)$ and $b\circ q(\mathbf m_i)$.
Therefore $\lambda_j$   converges to a point $(z_j)$ in the ray compactification such that $z_j$ is a weighted union of  $\mu|\Sigma_j$ and either $y_j$ or $\mu_j'$.
Therefore,  we have $\iota(\mu|\Sigma_j, z_j)=0$ if $z_j$ lies inside $\ml(\Sigma_j)$ and is represented by $\nu_j \in \ml(\Sigma_j)$ with $z_j=[\nu_j]$ and $\iota(\nu_j, \mu|\Sigma_j)=0$ if $z_j$ lies in $\pl(\Sigma_j)$.
Since $[\mu] \neq [\Lambda]$ and $\Lambda|\Sigma_j$ is arational, we have $(z_j)\neq [\Lambda]$, and we are led to a contradiction.
Thus we have shown that by taking $U_0$ to be sufficiently small, the condition (4) also holds.
This completes the proof.
\end{proof}

To define a locally degree-1 map $\cF \colon \ct \to \cml$, and a homotopy from $\cb$ to $\cF$, we shall first define a map $F\colon \teich(\Sigma_-) \times \teich(\Sigma_+) \to \ml(\Sigma_- )\times \ml(\Sigma_+)$ which induces $\cF$.

\begin{definition}
\label{F}
Let $U_0, U_1, U_2 \subset \ml(\Sigma_-)\times  \ml(\Sigma_+)$ and $V_1, V_2 \subset  \teich(\Sigma_-) \times \teich(\Sigma_+)$ be open sets  given in \cref{three rings}, and let $\tau$ be the train track given in \cref{train track}.
Let $F \colon \teich(\Sigma_-) \times \teich(\Sigma_+) \to \ml(\Sigma_- )\times \ml(\Sigma_+)$ be a continuous map defined as follows:
\begin{enumerate}[(i)]
\item For $x$ outside $V_2$, we define $F(x)=b \circ q(x)$.
\item For $x \in V_1$, we define $F(x)=\ce(x)$.
\item For $x \in V_2 \setminus V_1$, letting $t(x)$ be $\displaystyle \frac{d_\tau(x, \overline{V_1})}{d_\tau(x, \overline{V_1})+d_\tau(x, \overline{V_2^c})}$,
we define $F(x)$ to be the point dividing the segment connecting $\ce(x)$ and $b\circ q(x)$ internally by $t(x): 1-t(x)$ in the weight system coordinates of $\tau$.
\end{enumerate}
\end{definition}
\begin{lemma}
\label{proper F}
For the map $F$ defined above, there is no sequence in $\teich(\Sigma_-) \times \teich(\Sigma_+)$ diverging to infinity, whose image under $F$ has a subsequence converging in $\ml(\Sigma_-)\times \ml(\Sigma_+)\setminus D$.
\end{lemma}
%q(m_i)は収束すると言っておく．
\begin{proof}
Since both $\ce$ and $b \circ q$ are proper, $\ce$ as a map to $\ml(\Sigma_-) \times  \ml(\Sigma_+)$ and $b\circ q$ as a map to $\ml(\Sigma_-) \times \ml(\Sigma_+) \setminus D$, we have only to consider the case when $\{\mathbf m_i\}$ lies in  $V_2 \setminus V_1$.
Let $\{\mathbf m_i\}$ be a sequence in $V_2 \setminus V_1$ which does not have a convergent subsequence.
We can assume that it converges in the Thurston compactification product to a point $(y_j)$ in $\pl$, passing to a subsequence.
%これはemptyではないが，\Sigma_j上にはないようなjもあるかもしれない．
We need to show that $\{F(\mathbf m_i)\}$ does not have a convergent subsequence in $\ml(\Sigma_- ) \times \ml(\Sigma_+) \setminus D$.

Recall from \cref{three rings} that both $\{\mathcal E(\mathbf m_i)\}$ and $\{b\circ q(\mathbf m_i)\}$ lie in $U_2$.
Since $\ce$ is a proper map to $\ml(\Sigma_-)  \times \ml(\Sigma_+)$, the sequence $\{\ce(\mathbf m_i)\}$ diverges to infinity, necessarily within $U_2$.
If $\{b\circ q(\mathbf m_i)\}$ also diverges to infinity (within $U_2$), the segment connecting $\ce(\mathbf m_i)$ and $b \circ q(\mathbf m_i)$ also diverges to infinity within $U_2$ as $i \longrightarrow \infty$, and we are done.
It remains to deal with the case when $\{b\circ q(\mathbf m_i)\}$ converges to a point $\nu$ in $D$ after passing to a subsequence.
By the part (1) of \cref{three rings}, $\nu$ must satisfy the condition (c) of \cref{main}, and hence contains a compact leaf with weight larger than or equal to $\pi$.
We denote the union of all such components of $\nu$ by $\nu_0$.
Let $\Sigma_k$ be a component of $\Sigma_- \sqcup \Sigma_+$ containing a component of $\nu_0$.
Then $y_k$ lies in $\pl(\Sigma_k)$ since otherwise the component of the convex core boundary corresponding to $\Sigma_k$ converges geometrically without giving rise to a new parabolic curve, as argued in the previous section.

We first remark the following, which was just a restatement of \cref{must converge}.

\begin{claim}
\label{must converge again}
Let $\phi_i \colon \pi_1(S) \to \pslc$ be a representation corresponding to $q(\mathbf m_i)$.
Then, $\{\phi_i\}$ converges to some $\psi \in \AH(S)$ as $i \longrightarrow \infty$.
\end{claim}
 
%If $\mu$ intersects $\nu_0$ transversely, $\len_{\mathbf m_i}(\nu_0)$ goes to infinity.
%This contradicts \cref{length to 0}.
%On the other hand, if $\ce(\mathbf m_i)$ is disjoint from $\nu_0$ for infinitely many $i$, then $\len_{\mathbf m_i}(\nu_0)$ cannot go to $0$.
%Therefore, $\ce(\mathbf m_i)$ has non-empty intersection with $\nu_0$.

Let $P_i$ be a shortest pants decomposition of $(\Sigma_- \sqcup \Sigma_+, \mathbf m_i)$.
Let $P_\infty$ be the Hausdorff limit of $\{P_i\}$ (after passing to a subsequence), which is a geodesic lamination.
Invoking an argument which we used in the preceding section, we can show the following.

%i(\varrho, \nu)=0を言っておく．
\begin{claim}
\label{H-limit}
Every minimal component of $P_\infty$ is a simple closed curve.
The geometric limit $M_\infty$ has neither a new geometrically infinite end (i.e. one not corresponding to that of $(\hyperbolic^3/\phi_i(\pi_1(S))_0$) nor a torus cusp.
\end{claim}
\begin{proof}
Suppose, seeking a contradiction, that $P_\infty$ has a minimal component $\varrho$ which is not a simple closed curve.
Let $S(\varrho) \subset S$ be the minimal supporting surface of $\varrho$.
Then $M_\infty$ must have a new geometrically infinite end which lifts to the algebraic limit.
(This follows from the fact in Minsky's model of $\hyperbolic^3/\phi_i(\pi_1(S))$, the subsurface$S(\varrho)$ must support a tight geodesic whose length goes to $\infty$ in this case.)
%, as defined in \cref{mss}.

%Then by applying \cref{vertical} to  $\{\phi_i|\pi_1(S(\varrho))\}$, we see that a subsequence of $\{\phi_i|\pi_1(S(\varrho))\}$ converges up to conjugation.

Now, we put a basepoint on the algebraic locus $f_\infty(S)$, and consider a geometric limit $M_\infty$ containing the algebraic limit of $\{\phi_i\}$.
%Then \cref{nearest} shows that $M_\infty$ has the nearest simply degenerate end  whose ending lamination is $\varrho$.
As in \cref{nearest sd end}, by taking a non-peripheral simple closed curve $d$ on $S(\varrho)$ with $\iota(d, \varrho)<\delta$ for the constant $\delta$ given in \cref{homotopic geodesic} with $\epsilon=1$, we are led to a contradiction.
Thus we have shown that every minimal component of $P_\infty$ is a simple closed curve.
This argument also shows that $M_\infty$ cannot have a nearest simply degenerate end, for such an end can be lifted to the algebraic limit by \cref{nearest}.

Next suppose that $M_\infty$ has a torus cusp.
%Since $P_\infty$ is a multi-curve, there is no nearest simply degenerate end for $M_\infty$.
Since there is no nearest simply degenerate end for $M_\infty$, if there are torus cusps, we can take a nearest one by \cref{nearest}.
By repeating the arguments of \cref{nearest torus}, which can be applied also to our situation as remarked there, we get a contradiction.

Since $M_\infty$ has neither a nearest simply degenerate end nor a nearest torus cusp, by \cref{nearest}, we see that  $M_\infty$ does not have a new geometrically infinite end.
\end{proof}
%If $P_i$ intersects $\nu_0$ transversely, then
%Therefore, we can assume that $\mu$ is a multi-curve.

%\begin{comments}
%$\nu_0$が$P_i$のlimitに含まれることを示す．
%\end{comments}
%Therefore $P_\infty$ is a multi-curve.
We define a subset $P^0_\infty$ of $P_\infty$ to be the subset consisting of simple closed curves  whose lengths with respect to $\mathbf m_i$ go to $0$ as $i \longrightarrow \infty$.

%\begin{claim}
%\label{no torus}
%The geometric limit $M_\infty$ does not have a torus cusp either.
%\end{claim}
%\begin{proof}
%Since $P_\infty$ is a multi-curve, there is no new geometrically infinite end for $M_\infty$.
%If there are torus cusps, we can take a nearest one by \cref{nearest}.
%By repeating the arguments of \cref{nearest torus}, which can be applied also to our situation as remarked there, we are done.
%\end{proof}
Since $\{\phi_i\}$ converges to $\psi \in \AH(S)$ as mentioned above, by \cref{strong}, the claim above implies that the convergence is strong.

\begin{claim}
\label{coincide}
The  multi-curve $P_\infty^0$ is contained in  $\nu_0$.
\end{claim}
\begin{proof}
Let $c$ be a component of $P_\infty^0$.
We regard $c$ as lying on a compact core of $\hyperbolic^3/\psi(\pi_1(S)))$.
Then $\len_{\mathbf m_i}(c)$ goes to $0$ by the definition of $P_\infty^0$, and hence it represents a core curve of a $\integers$-cusp neighbourhood lying outside the compact core.
Since $\{\phi_i\}$ converges to $\psi$ strongly and $M_\infty$ does not have a new geometrically infinite end by \cref{H-limit}, the $\integers$-cusp has geometrically finite ends on its both sides.
We can apply \cref{new parabolic} to see that the bending angle along $c$ converges to $\pi$, and hence $c$ is contained in $\nu_0$.
%Since $\len_{m_i}(c) \longrightarrow 0$, we also have $\len_{\phi_i}(c) \longrightarrow 0$.
%Therefore, $\psi(c)$ is a parabolic element, which must correspond to a $\integers$-cusp in $M_\infty$ by \cref{no torus}.
%By \cref{no torus}, the $\integers$-cusp 
%Since $P_\infty$ does not contain a non-simple closed curve as a minimal component, a geometric limit $M_\infty$ of $\hyperbolic^3/\phi_i(\pi_1(S))$ with base point at the algebraic locus $f_\infty(S)$ cannot  have a geometrically infinite end.
%Since $\len_{\mathbf m_i}(c)$ goes to $0$,   the curve $c$ corresponds to a $\integers$-cusp in $M_\infty$ both of whose sides correspond to geometrically finite ends, and hence the curve corresponding to $c$ has bending angle $\pi$ there.
%By \cref{OS}-(c), it follows that  $b \circ q(\mathbf m_i)$ tends to $\pi$, and hence $c$ is contained in $\nu_0$.
%
%On the other hand, by \cref{length to 0}, every component of $\nu_0$ has length going to $0$ with respect to $\mathbf m_i$.
%Therefore $\nu_0$ is contained in $P_\infty^0$.
%Therefore, it suffices to show that every component of $\nu_0$ is contained in $P_\infty$.
%every minimal component of $P_\infty$ is a simple closed curve.
%Therefore, $P^_\infty$ consists of simple closed curves and infinite geodesics spiralling around them.
%
%
%Suppose that $c$ is not a compact 
\end{proof}

We can further see the following.
\begin{claim}
\label{twist}
Let $C$ be a minimal component of $P_\infty$ not contained in $P^0_\infty$.
If $P_\infty$ does not contain any other minimal component homotopic to $C$ in $S \times [0,1]$,
then the twisting parameter along $C$ of $\mathbf m_i$ is bounded as $i \longrightarrow \infty$.
If $P_\infty$ has two distinct minimal components $C$ and $C'$ homotopic to each other in $S\times [0,1]$, then the difference of the twisting parameters of $\mathbf m_i$ along $C$ and $C'$ is bounded as $i \longrightarrow \infty$.
\end{claim}
\begin{proof}
Suppose first that there is no other component in $P_\infty$ homotopic to $C$ in $S \times I$.
If the twisting parameter of $\mathbf m_i$ along $C$ goes to $\infty$ after passing to a subsequence, then by the part (1) of \cref{torus}, the geometric limit $M_\infty$ has a corresponding torus cusp.
This contradicts \cref{H-limit}.
%Therefore the closed geodesic representing $\phi_i(C)$ in $\hyperbolic^3/q(\mathbf m_i)$ lies in the interior of the convex core.
%This means that $b \circ q(\mathbf m_i)$ does not contain $C$ as a leaf.
%Since the geometric limit $M_\infty$ does not have 
%Therefore, in the geometric limit, there is a boundary component of the convex core above the torus cusp in such a way that the vertical projection of the bending lamination intersects $C$ essentially.
%This implies that the support of $b \circ q(\mathbf m_i)$ has a Hausdorff limit containing $C$ and leaves spiralling around it in such a way that the measure of a fixed arc transverse to $C$ goes to $\infty$.
%In particular, $b \circ q(\mathbf m_i)$ diverges in $\ml(\Sigma_-) \times \ml(\Sigma_+)$, contradicting our assumption.
%
Also in the latter case when $C$ is homotopic to $C'$ in $S \times I$,  
%a surface representing $\phi_i$ may wrap around the Margulis tube corresponding to $\phi_i(C)$. 
%Still 
by the same argument as above involving \cref{torus}, we see that if the difference of twisting parameters of $\mathbf m_i$ goes to $\infty$, then $M_\infty$ must have a torus cusp, contradicting \cref{H-limit}.%Thus we have completed the proof.
 \end{proof}
 
 This claim implies the following.
 
 \begin{claim}
 \label{in 0}
 The projective lamination $y_k$ is supported in $P^0_\infty$.
 \end{claim}
 \begin{proof}
By the definition of the topology of the Thurston compactification product, we see that $|y_k|$ is contained in $P_\infty$.
By \cref{H-limit}, every minimal component of $P_\infty$ is a simple closed curve.
If $|y_j|$ has a component $C$ contained in $P_\infty$ but not in $P^0_\infty$, then the twist parameter of $\mathbf m_i$ around $C$ must go to either $\infty$ or $-\infty$.
This contradicts \cref{twist} unless there is another component of $P_\infty$ homotopic to $C$ in $S \times I$.
If there is such a component $C'$, then the support of some coordinate $y_l$ of $(y_j)$ must  contain $C'$ as a component by \cref{twist} again.
On the other hand, $(y_j)$ is contained in $U_2$, and hence cannot have two components of projective laminations homotopic to each other in $S \times I$.
This is a contradiction.
 \end{proof}

Having proved these claims, we can now complete the proof of \cref{proper F}.
Since $|y_k|$ is contained in $P^0_\infty$ by \cref{in 0}, it lies in $\nu_0$ by \cref{coincide}.
Therefore the segment connecting $\mathcal E(\mathbf m_i)$ and $b \circ q(\mathbf m_i)$ converges uniformly on any compact set to a ray entirely lying in $D$.
%This shows that $F$ is also proper in $V_2 \setminus V_1$, and 
This completes the proof.
\end{proof}

Now we can define $\cF$ which we mentioned before.
By \cref{proper F}, $F \colon \teich(\Sigma_-) \times \teich(\Sigma_+) \to \ml(\Sigma_-) \times \ml(\Sigma_+)$ induces a unique continuous map $\cF\colon \ct\to \cml$ similarly to \cref{compactification}.
%We denote this map by $\cF$.
We shall show that $\cF$ has degree $1$.
We use the symbols $\cE\colon \ct \to \cml$  to denote the map induced by $\ce$.

\begin{proposition}
\label{F degree}
The map $\cF \colon \ct \to \cml$ has degree $1$.
\end{proposition}
\begin{proof}
Let $\cU_0$ be the open set in $\cml$ corresponding to  $U_0 \setminus D$, where $U_0$ is the truncated cone neighbourhood  in \cref{three rings}.
We shall show that  the restriction of $\cF$ to  $\cF^{-1}(\cU_0)$ has degree $1$, which immediately implies that $\cF$ has degree $1$.

Since $F|\ce^{-1}(U_0)$ coincides with $\ce|\ce^{-1}(U_0)$, it is a homeomorphism to its image.
We have only to show that there is no point outside $\ce^{-1}(U_0)$  which is mapped into $U_0$.
Since $F|V_1$ coincides with $\ce|V_1$, there are no points in $V_1 \setminus \ce^{-1}(U_0)$ mapped into $U_0$ by $F$.
On the other hand, the conditions (3) and (4) in \cref{three rings} guarantee that no points outside $V_1$ are mapped into $U_0$ by $F$.
This completes the proof.
\end{proof}

The next step is to show that there is a homotopy between $\cb$ and $\cF$.
We shall first construct a homotopy between $b \circ q$ and $F$.

%\begin{definition}
We define a homotopy $H \colon \teich(\Sigma_-) \times \teich(\Sigma_+) \times [0,1] \to \ml(\Sigma_-) \times \ml(\Sigma_+)$ from $b \circ q$ to $F$ as follows.
\begin{enumerate}[(1)]
\item For  any $x$ outside $V_2$ and any $s \in [0,1]$, $H(x , s)$ is defined to be $b\circ q (x)$.
\item For any $x$ in $V_2$ and any $s \in [0,1]$,  $H(x,s)$ is defined to be the point dividing the segment connecting $b\circ q(x)$ and $F(x)$ internally by $s:1-s$.
\end{enumerate}
%\end{definition}

\begin{lemma}
For any sequence $\{t_i\in [0,1]\}$ and any $\{\mathbf m_i\} \subset \teich(\Sigma_-) \times \teich(\Sigma_+)$ diverging to infinity, $\{H(\mathbf m_i, t_i)\}$ cannot have a  subsequence converging in $\ml(\Sigma_-)\times \ml(\Sigma_+) \setminus D$ as $i \longrightarrow \infty$.
\end{lemma}
\begin{proof}
The proof is quite similar to that of \cref{proper F}.
We can assume that $\{\mathbf m_i\}$ either lies outside $V_2$,   in $V_2 \setminus V_1$, or in $V_1$, passing to a subsequence.
In the case when $\{\mathbf m_i\}$ lies outside $V_2$, the statement follows immediately from the properness of $b \circ q$ as a map to $\ml(\Sigma_-)\times \ml(\Sigma_+) \setminus D$.
In the case when $\{\mathbf m_i\}$ lies in either  $V_2 \setminus V_1$ or  $V_1$, the point $H(\mathbf m_i, t_i)$ lies in the segment connecting $b\circ q(\mathbf m_i)$ and $\ce(\mathbf m_i)$.
By the same argument as in the proof of \cref{proper F}, such a segment converges to a ray on $D$ as $i \longrightarrow \infty$.
This shows that $\{H(\mathbf m_i, t_i)\}$ cannot have a subsequence converging in $\ml(\Sigma_-) \times \ml(\Sigma_+) \setminus D$.
\end{proof}

This lemma immediately implies the following.
\begin{corollary}
The homotopy $H$ induces a homotopy $\ch \colon \ct \to \cml$ between $\cb$ to $\cE$.
\end{corollary}

Combining this corollary with \cref{F degree}, we conclude that $\cb$ has degree $1$, which in turn implies that $b \circ q$ has degree $1$.
This completes the proof of the latter half of \cref{proper degree 1}.
%In this section, we shall prove the compactness part (2) of \cref{main} in the case when $\lambda =\emptyset$.
%We state this part as a theorem:
%\begin{theorem}
%\label{qf}
%Let $\QF_{\lambda, \mu}$ be the set of quasi-Fuchsian representations  (modulo conjugacy) of $\pi_1(S)$  realising $\lambda$ and $\nu$. 
%Then, $\QF_{\lambda, \mu}$ is a compact subset of $\QF$.
%\end{theorem}
%
%To prove \cref{qf} by contradiction, we need to analyse possible form of geometric limits of sequences in $\QF_{\lambda, \mu}$.
\bibliographystyle{acm}
 \bibliography{convex-core}
\end{document}